\documentclass[11pt]{article}

\usepackage{amsfonts,amsthm,amssymb,amsmath}
\usepackage[title,titletoc,toc]{appendix}
\usepackage[utf8]{inputenc}
\usepackage[affil-it]{authblk}
\usepackage{xcolor}
\usepackage{tikz}

\usepackage[T1]{fontenc}
\usepackage{stmaryrd}

\newtheorem{theorem}{Theorem}[section]

\newtheorem{proposition}[theorem]{Proposition}
\newtheorem{remark}[theorem]{Remark}

\newtheorem{lemma}[theorem]{Lemma}

\newtheorem{corollary}[theorem]{Corollary}
\newtheorem{definition}[theorem]{Definition}

\usepackage{scalerel,stackengine}
\stackMath
\newcommand\reallywidehat[1]{%
\savestack{\tmpbox}{\stretchto{%
  \scaleto{ \scalerel*[\widthof{\ensuremath{#1}}]{\kern.1pt\mathchar"0362\kern.1pt} {\rule{0ex}{\textheight}} }{\textheight}}{2.4ex}}
\stackon[-6.9pt]{#1}{\tmpbox}}

\begin{document}

\title{New Stochastic Fubini Theorems
}

\author{Tahir Choulli\\ 
Mathematical and Statistical Sciences Dept.\\
University of Alberta, Edmonton, Canada\\
tchoulli@ualberta.ca
\and 
Martin Schweizer\\ 

ETH Z\"urich\\
Departement Mathematik\\
ETH-Zentrum, HG G 51.2\\
R\"amistrasse 101\\
 CH -- 8092 Z\"urich\\
martin.schweizer@math.ethz.ch}

\maketitle

\begin{abstract}
The classic stochastic Fubini theorem says that if one stochastically integrates
with respect to a semimartingale $S$ an $\eta(dz)$-mixture of $z$-parametrized integrands
$\psi^z$, the result is just the $\eta(dz)$-mixture of the individual $z$-parametrized
stochastic integrals $
\int\psi^z{d}S.$ But if one wants to use such a result for the study
of Volterra semimartingales of the form $ X_t =\int_0^t \Psi_{t,s}dS_s, t \geq0,$ the classic assumption that one has a fixed measure $\eta$ is too restrictive; the mixture over
the integrands needs to be taken instead with respect to a stochastic kernel on the parameter space. To handle that situation and prove a corresponding
new stochastic Fubini theorem, we introduce a new notion of measure-valued stochastic integration with respect to a general multidimensional semimartingale. As an application, we show how this allows to handle a class of quite general stochastic Volterra semimartingales.
\end{abstract}


\noindent{\bf MSC 2000 Classification Numbers:} 60H05, 28B05, 60G48\\

\section{Introduction}
The standard setup for a stochastic Fubini theorem is that one has a semimartingale $S$, a
parametric family $\psi^z$ of processes indexed by $z$ from some space $\cal{Z}$, and a (fixed) measure $\eta$ on $\cal{Z}$. For each $z$, the process $\psi^z$ is assumed to be $S$-integrable, and so is the mixture $\bar{\psi}^{\eta}:=\int\psi^z\eta(dz)$. The typical classic stochastic Fubini theorem then has the form
\begin{equation*}
\int_{(0,\cdot]}\left(\int_{\cal{Z}}\psi^z\eta(dz)\right)dS_t=\int_{(0,\cdot]}\bar{\psi}^{\eta}_t{d}S_t=\int_{\cal{Z}}\left(\int_{(0,\cdot]} \psi^z_t{d}S_t\right)\eta(dz);\tag{0.1}
\end{equation*}
see for instance Th\'eor\`eme 5.44 of Jacod (1979), Lebedev (1995) [which corrects an error in
Jacod (1979)], Theorem IV.65 of Protter (2005), or Bichteler/Lin (1995) [which extends the
result in (the first, 1990 edition of) Protter (2005) from an $L^2$- to an $L^1$-setting]. A good
overview with very weak assumptions has recently been given by Veraar (2012).

The Fubini property (0.1) says that if one stochastically integrates a mixture (over $z$,
with respect to $\eta$) of integrands, the result is given by the analogous mixture of the individual
stochastic integrals. This looks very general, but it actually imposes one very restrictive
assumption --- one has a single fixed measure ? which mixes over the parameter $z$, and
which does not depend on the randomness $\omega$ (nor on time $t$). Some very natural applications
like the study of Volterra semimartingales of the form $X_t=\int_0^t\Psi_{s,t}dS_s$, $t\geq 0$, lead almost
immediately outside the scope of the above standard setup. In this paper, we therefore study
what can be done if one replaces $\eta(dz)$ by a kernel  $\rho(\omega,t,dz)$ from $\Omega\times(0,\infty)\times{\cal{P}})$ to $\cal{Z}$. One can then still easily make sense of the left-hand side of (0.1) by just mixing the integrands
with $\rho$ instead of $\eta$, so that one considers $\bar{\psi}^{\rho}:=\int\psi^z_t\rho_t(dz)$; but it is a priori not clear how the result can be expressed on the right-hand side sonce it makes no sense (and is not even well defined) to take there a mixture of  $\int\psi^z_t{d}S_t$ with $\rho_t(dz)$.

To answer this question, we have to take a new viewpoint and develop a new theory. Our
solution will be to view $\varphi_t(dz):= \psi^z_t\rho_t(dz)$ as a {\it measure-valued stochastic integral} of $\varphi$ with respect to $S$, and to obtain a {\it weak}* {\it stochastic
Fubini theorem} by plugging in test functions $f$ or indicators $I_D$ of sets $D\subseteq{\cal{Z}}$. This will
allow us to make sense of a generalized form of (0.1). We explain below in more detail why
this needs in particular a new stochastic integral.

Our approach combines and generalises two ideas from the existing literature in a nontrivial
way. The idea of viewing stochastic Fubini theorems via more abstract integrands
can be traced back to van Neerven/Veraar (2006) who proved a stochastic Fubini theorem
for an infinite-dimensional integrator $S$ (a cylindrical Brownian motion). But they still work
with a fixed reference measure $\eta$ on $\cal{Z}$, and so their integrands have values in $L^1({\cal{Z}},\eta)$,
since one can identify a finite measure $m\ll\eta$ with its Radon--Nikod\'ym derivative 
${{dm}\over{d\eta}}$.
Moreover, their whole approach does not extend substantially beyond the case of Brownian motion as integrator. For more general measure-valued integrands, Bj\"ork/Di Masi/Kabanov/Runggaldier (1997) develop a stochastic integration theory with respect to a function-valued
integrator $S$, with the resulting integral being real-valued. The construction and setup for
our measure-valued stochastic integral is strongly inspired by that work. However, Bj\"ork/Di
Masi/Kabanov/Runggaldier (1997) did not look at the issue of more general stochastic Fubini
theorems at all, and their integral processes are real-valued which makes their definition and
construction reasonably straightforward.

One important application as well as motivation is the study of Volterra semimartingales
of the form $X_t=\int_0^t\Psi_{s,t}dS_s$ with a two-parameter process $\Psi=(\Psi_{t,s})$. If $\Psi$ is, with respect to
$s$, predictable and $S$-integrable for each $t$, including on the diagonal $s=t$, and, with respect
to $t$, of finite and $S$-integrable variation, we show that $X$ is again a semimartingale, and that
we can say more about its decomposition. This substantially generalizes a result by Protter
(1985) which was obtained under the restrictive condition that each $\Psi_{t,s}$ is, with respect to $t$,
absolutely continuous with respect to Lebesgue measure. That assumption allows to reduce
the question to a classic stochastic Fubini theorem, which is no longer possible if absolute
continuity (with respect to a fixed reference measure) is generalized to finite variation.

Our results rely on a stochastic integral of measure-valued processes with respect to a
general $\mathbb{R}^d$-valued semimartingale, with the integral also taking values in a space of measures.
This is an unusual setup and to the best of our knowledge not covered by existing literature.
One often finds stochastic integrals where either integrand or integrator are measure-valued
and integrator respectively integrand take values in a corresponding dual space; see for instance
Mikulevicius/Rozovskii (1998). We also mention the paper by Krylov (2011) who
proves an It\^o-Wentzell formula for distribution-valued processes by plugging in test functions
and using or extending classic results. The crucial point here is that the integral processes in
all these cases are still real-valued. If one looks instead at integrals with values in a Banach
space, the construction of a stochastic integral (even with respect to Brownian motion) typically
becomes nontrivial and depends on the geometry of the Banach space; see for example
Dettweiler (1985). Some recent papers have managed to overcome this difficulty by reformulating
the It\^o isometry property via an operator-theoretic approach. But this has been done
only for the case where the integrator is a (possibly infinite-dimensional) Brownian motion,
and the approach also crucially exploits the underlying Gaussian structure; see for instance
van Neerven/Veraar/Weis (2007). A recent overview of this theory with many references
can be found in the survey paper by van Neerven/Veraar/Weis (2013). For semimartingale
integrators, corresponding results do not seem to exist yet.

This paper is organized as follows. We present the basic
setup and some preliminary results in Section 2. Section 3 develops the stochastic integral of measure-valued processes $\varphi$ with respect to the semimartingale $S$, and this integral is vital for our new stochastic Fubini theorems. Section 4 contains new stochastic Fubini
theorems. We start with the case when $S$ is a locally square integrable local martingale, in two variants for different applications. Then, we extend the results to the semimartingale case afterwards. Section 5 shows how these results can be used to study the above Volterra semimartingale $X$, and we also explain how they help to obtain a decomposition for $X$. Section 6 discusses the connection between our new stochastic Fubini theorem and the classical stochastic Fubini theorem, while Section 7 illustrates an example. 

\section{Preliminaries}
In this section, we collect some basic notations and concepts.

\subsection{Notations}
Let $(\Omega, \mathcal{F}, P)$ be a probability space with a filtration $\mathbb{F}=\left(\mathcal{F}_{t}\right)_{t \geq 0}$ satisfying the usual conditions of right-continuity and completeness. Throughout the paper, we set $\mathcal{F}_{\infty}:=\bigvee_{t \geq 0} \mathcal{F}_{t}=\sigma\left(\bigcup_{t \geq 0} \mathcal{F}_{t}\right)$ and denote by $\mathcal{P}$ the predictable $\sigma$-field on $\bar{\Omega}:=\Omega \times(0, \infty)$. Standard terminology and results from stochastic calculus can be found in Dellacherie/Meyer [9, Chapters VI-VIII] and Jacod/Shiryaev [16, Chapter I]. Identifying as usual processes that are indistinguishable, we define $$\mathcal{R}^{0}:=\left\{\mbox{all real-valued adapted RCLL processes}\ Z=\left(Z_{t}\right)_{t \geq 0}\right\}.$$ For any $Z \in \mathcal{R}^{0}$, we set $Z_{t}^{*}:=\sup _{0 \leq s \leq t}\left|Z_{s}\right|$ and $Z_{t-}^{*}:=\sup _{0 \leq s<t}\left|Z_{s}\right|$ for $t \geq 0$, with $Z_{\infty}^{*}=Z_{\infty-}^{*}:=\sup _{t \geq 0}\left|Z_{t}\right|$. Note that $Z^{*}=\left(Z_{t}^{*}\right)_{t \geq 0}$ is in $\mathcal{R}^{0}$, but $Z_{\infty}^{*}$ may take the value $+\infty$. We equip $\mathcal{R}^{0}$ with the topology of ucp-convergence (i.e., $Z^{n} \rightarrow Z$ if $\left(Z^{n}-Z\right)_{t}^{*} \rightarrow 0$ in probability for every $t \geq 0$ ). We also need the Banach space 
$$\mathcal{R}^{2}:=\left\{Z \in \mathcal{R}^{0}:\|Z\|_{\mathcal{R}^{2}}:=\left\|Z_{\infty}^{*}\right\|_{L^{2}}<\infty\right\}.$$
For any nonnegative increasing process $A$ in $\mathcal{R}^{0}$ and any (possibly multidimensional) product-measurable $H=\left(H_{t}\right)_{t>0}$ on $\bar{\Omega}$, we define the process $\mathcal{D}(H ; A)$ by
$$
\mathcal{D}_{t}(H ; A):=\int_{0}^{t}\left|H_{r}\right|^{2} \mathrm{~d} A_{r}:=\int_{(0, t]}\left|H_{r}\right|^{2} \mathrm{~d} A_{r} \in[0, \infty] \quad \text { for all } t \geq 0
$$
Note that $\mathcal{D}(H ; A)$ is always increasing, null at 0 and may take the value $+\infty$; it is in addition adapted if $H$ is progressively measurable (and in particular if $H$ is predictable). For any stopping time $\tau$, we define the measure $\mu_{\tau, A}$ on the product $\sigma$-field of $\bar{\Omega}$ by
$$
\mu_{\tau, A}(C):=E\left[A_{\tau-} \mathcal{D}_{\tau-}\left(I_{C} ; A\right)\right]=E\left[A_{\tau-}(\omega) \int_{0}^{\tau-(\omega)} I_{C}(\omega, r) \mathrm{d} A_{r}(\omega)\right]
$$
so that
$$
\|H\|_{L^{2}\left(\mu_{\tau, A}\right)}^{2}=E\left[A_{\tau-} \int_{0}^{\tau-}\left|H_{r}\right|^{2} \mathrm{~d} A_{r}\right]=E\left[A_{\tau-} \mathcal{D}_{\tau-}(H ; A)\right]
$$
Note that both $A$ and $\mathcal{D}(1 ; A)=A-A_{0}$ are in $\mathcal{R}^{0}$. Therefore the product $A \mathcal{D}(1 ; A)$ is prelocally bounded, and we can find a localizing sequence $\left(\tau_{M}\right)_{M \in \mathbb{N}}$ of stopping times such that $\mu_{\tau_{M}, A}$ is a finite measure for each $M \in \mathbb{N}$. Note that if $H$ is bounded by a constant $C$, say, then $\mathcal{D}(H ; A) \leq C^{2} A$ so that any bounded process is in $L^{2}\left(\mu_{\tau, A}\right)$ as soon as $A_{\tau-} \in L^{2}$.
\subsection{Semimartingales and stochastic integrals}
For semimartingales and stochastic integration, we use the approach due to M\'etivier and Pellaumail as presented in the textbook by M\'etivier [21]. First, an $\mathbb{R}^{d}$-valued simple predictable process is of the form
\begin{equation*}
H=\sum_{\ell=0}^{L-1} h_{\ell} I_{A_{\ell} \times\left(t_{\ell}, t_{\ell+1}\right]} \tag{2.1}
\end{equation*}
where $L \in \mathbb{N}, 0 \leq t_{0}<t_{1}<\cdots<t_{L}<\infty, A_{\ell} \in \mathcal{F}_{t_{\ell}}$ and $h_{\ell}$ is $\mathbb{R}^{d}$-valued, bounded and $\mathcal{F}_{t_{\ell}}$-measurable. If the $h_{\ell}$ are nonrandom, we call $H$ very simple and denote by $\mathcal{H}$ the family of all very simple predictable processes. For any $\mathbb{R}^{d}$-valued adapted RCLL process $S=\left(S_{t}\right)_{t \geq 0}$ and $H$ as in (2.1), we define the stochastic integral process $H \cdot S=\int H \mathrm{~d} S$ by
$$
H \cdot S_{t}:=(H \cdot S)_{t}:=\sum_{\ell=0}^{L-1} I_{A_{\ell}} h_{\ell}^{\top}\left(S_{t_{\ell+1} \wedge t}-S_{t_{\ell} \wedge t}\right)=\sum_{\ell=0}^{L-1} I_{A_{\ell}} h_{\ell}^{\top}\left(S_{t_{\ell+1}}^{t}-S_{t_{\ell}}^{t}\right), \quad t \geq 0
$$
It is clear that $H \cdot S$ is in $\mathcal{R}^{0}$ and null at 0 . A control process is a nonnegative increasing process $V=\left(V_{t}\right)_{t \geq 0}$ in $\mathcal{R}^{0}$ such that for any stopping time $\tau$,
\begin{equation*}
E\left[\sup _{0 \leq t<\tau}\left|\int_{0}^{t} H_{r} \mathrm{~d} S_{r}\right|^{2}\right] \leq E\left[V_{\tau-} \int_{0}^{\tau-}\left|H_{r}\right|^{2} \mathrm{~d} V_{r}\right] \quad \text { for any } H \in \mathcal{H} \text {. } \tag{2.2}
\end{equation*}
We denote by $\mathcal{V}(S)$ the family of all control processes for $S$. By Métivier [21, Theorem 23.14], $S$ is a semimartingale if and only if it admits a control process, i.e., $\mathcal{V}(S) \neq \emptyset$. Indeed, with $\mathbb{H}=\mathbb{R}^{d}$, the "if" part follows from part $\left(2^{\circ}\right)$ of that theorem, and the "only if" part follows from part $\left(1^{\circ}\right)$ with $\mathbb{G}=\mathbb{R}^{d}$, noting that one can easily extend (2.2) from very simple to simple predictable processes.

Now suppose $S$ is an $\mathbb{R}^{d}$-valued semimartingale. Using the notations from Section 2.1 and choosing $V \in \mathcal{V}(S) \neq \emptyset$, we can rewrite (2.2) compactly, for any stopping time $\tau$, as
\begin{equation*}
\left\|H \cdot S^{\tau-}\right\|_{\mathcal{R}^{2}}=\left\|(H \cdot S)^{\tau-}\right\|_{\mathcal{R}^{2}} \leq\|H\|_{L^{2}\left(\mu_{\tau, V}\right)} \quad \text { for any } H \in \mathcal{H} \tag{2.3}
\end{equation*}
If $\tau$ is such that $V_{\tau-} \in L^{2}$ (or, equivalently, $\mu_{\tau, V}$ is a finite measure), then by [21, Sections 24.1 and 26.1], any $\mathbb{R}^{d}$-valued predictable process $H$ with $\|H\|_{L^{2}\left(\mu_{\tau, V}\right)}<\infty$ is integrable with respect to $S$ on $\llbracket 0, \tau \llbracket$ so that $(H \cdot S)^{\tau-}$ is well defined and a real-valued semimartingale. As $S=S^{\tau-}$ on $\llbracket 0, \tau \llbracket$, we can also write $(H \cdot S)^{\tau-}=H \cdot S^{\tau-}$ and view this equivalently either as the stochastic integral of $H$ with respect to $S$ on $\llbracket 0, \tau \llbracket$, or as the stochastic integral of $H$ with respect to $S^{\tau-}$ on $\llbracket 0, \infty \rrbracket=[0, \infty) \times \Omega$. From the construction in [21], we also have, extending (2.3), that (again for $V_{\tau-} \in L^{2}$)
\begin{equation*}
\left\|H \cdot S^{\tau-}\right\|_{\mathcal{R}^{2}} \leq\|H\|_{L^{2}\left(\mu_{\tau, V}\right)}\ \text { for any } \mathbb{R}^{d} \text {-valued predictable } H \in L^{2}\left(\mu_{\tau, V}\right) \tag{2.4}
\end{equation*}
To get rid of the stopping time $\tau$, fix a control process $V \in \mathcal{V}(S)$. Take any $\mathbb{R}^{d}$-valued predictable process $H$ such that $\mathcal{D}(H ; V)=\int\left|H_{r}\right|^{2} \mathrm{~d} V_{r}$ is finite-valued (i.e., in $\mathcal{R}^{0}$ ) and note that both $V$ and $\mathcal{D}(H ; V)$ are in $\mathcal{R}^{0}$ and hence prelocally bounded. Choose a localising sequence $\left(\tau_{M}\right)_{M \in \mathbb{N}}$ with $V_{\tau_{M}-} \in L^{2}$ and $H \in L^{2}\left(\mu_{\tau_{M}, V}\right)$, for each $M \in \mathbb{N}$. This allows us to define the stochastic integral $H \cdot S=\int H \mathrm{~d} S$ on every stochastic interval $\llbracket 0, \tau_{M} \llbracket$ and hence on $\llbracket 0, \infty \rrbracket=[0, \infty) \times \Omega$ by pasting together. It is shown in $[21$, Section 24.1$]$ that this stochastic integral does not depend on the choice of localising sequence, nor on the choice of the control process $V \in \mathcal{V}(S)$.
\subsection{Measure-valued processes}
Fix a compact metric space $\mathcal{K}$ equipped with its Borel $\sigma$-field $\mathcal{B}(\mathcal{K})$, denote by $\mathbb{M}:=\mathbb{M}(\mathcal{K})$ the space of signed and finite measures $m$ on $\mathcal{B}(\mathcal{K})$ and recall that each $m \in \mathbb{M}$ can be written as $m=m^{+}-m^{-}$for two unique finite measures $m^{ \pm}$on $\mathcal{B}(\mathcal{K})$ with disjoint supports. The total variation measure of $m$ is $|m|=m^{+}+m^{-}$. Write $C(\mathcal{K}):=C(\mathcal{K} ; \mathbb{R})$ for the space of continuous functions $f: \mathcal{K} \rightarrow \mathbb{R}$ with the sup-norm $\|f\|_{\infty}$ and denote the closed unit ball in $C(\mathcal{K})$ by $U_{1}:=\left\{f \in C(\mathcal{K}):\|f\|_{\infty} \leq 1\right\}$. The integral of $f$ with respect to $m$ is
$$
m(f):=\int f \mathrm{~d} m:=\int_{\mathcal{K}} f(z) m(\mathrm{~d} z) \quad \text { for } f \in C(\mathcal{K}) \text { and } m \in \mathbb{M}
$$
The variation norm of $m \in \mathbb{M}$ (not to be confused with $|m|$) is then given by
$$
\|m\|_{\mathrm{var}}:=\sup \left\{m(f): f \in U_{1}\right\}=|m|(\mathcal{K}) .
$$
We take on $\mathbb{M}$ the $\sigma$-field $\mathcal{M}$ generated by the weak* topology, i.e. by all the mappings $m \mapsto m(f)$ with $f \in C(\mathcal{K})$. For $d \in \mathbb{N}$ and $m=\left(m^{i}\right)_{i=1, \ldots, d} \in \mathbb{M}^{d}$, we set
$$
\|m\|_{\mathrm{var}}:=\left(\left\|m^{i}\right\|_{\mathrm{var}}\right)_{i=1, \ldots, d}, \quad m(f):=\left(m^{i}(f)\right)_{i=1, \ldots, d}
$$
Finally, we denote by $b B_{1}(\mathcal{K})$ the space of all bounded $g: \mathcal{K} \rightarrow \mathbb{R}$ of Baire class 1 . Recall that $g \in b B_{1}(\mathcal{K})$ if and only if $g$ is a pointwise limit of continuous functions; see e.g. Kuratowski [19, Theorems II.VIII. 1 and II.VIII.7].

For ease of reference, we list here some properties used later and give the corresponding references to the textbook by Aliprantis/Border [2]. Because $\mathcal{K}$ is compact and metric, it is Hausdorff and Polish by [2, Theorem 3.28 and Lemma 3.26], and $C(\mathcal{K})$ is separable for the sup-norm by $[2$, Lemma 3.99]. In turn, separability of $C(\mathcal{K})$ implies that the closed unit ball of its dual space $\mathbb{M}(\mathcal{K})=\mathbb{M}$ is metrisable for the weak* topology by [2, Theorem 6.30], and the same also holds for $\mathbb{M}^{d}=\left((C(\mathcal{K}))^{d}\right)^{*}$. Because $\mathcal{K}$ is Polish, every finite (nonnegative) measure $\mu$ on $\mathcal{B}(\mathcal{K})$ is regular by $\left[2\right.$, Theorem 12.7], and therefore $C(\mathcal{K})$ is dense in $L^{p}(\mu)$ for every $p \in[1, \infty)$, but not for $p=\infty$; see $[2$, Theorem 13.9].

We want to use (signed) measure-valued processes as integrands for a stochastic integral. A process $\varphi=\left(\varphi_{t}\right)_{t \geq 0}$ on $(\Omega, \mathcal{F}, P)$ with values in $\mathbb{M}^{d}$ is weak $k^{*}$ predictable if the $\mathbb{R}^{d}$-valued process $\varphi(f)=\left(\varphi_{t}(f)\right)_{t \geq 0}$ is predictable for each $f \in C(\mathcal{K})$. In view of the definition of the $\sigma$-field $\mathcal{M}$, this is equivalent to saying that $\varphi: \bar{\Omega} \rightarrow \mathbb{M}^{d}$ is $\mathcal{P}$ - $\mathcal{M}^{d}$-measurable. We remark for later use that the $\mathbb{R}^{d}$-valued process $\|\varphi\|_{\text {var }}=\left(\left\|\varphi_{t}\right\|_{\text {var }}\right)_{t \geq 0}$ is then predictable. Indeed, by separability of $C(\mathcal{K})$, we can take a countable dense subset $\underline{u}=\left(u_{k}\right)_{k \in \mathbb{N}}$ of $U_{1}$ and obtain $\sup _{f \in U_{1}} G(f)=\sup _{k \in \mathbb{N}} G\left(u_{k}\right)$ for any continuous function $G: C(\mathcal{K}) \rightarrow \mathbb{R}$. Taking $G(f)=\varphi_{t}^{i}(f)(\omega)$ for $i=1, \ldots, d$ thus yields for each $(\omega, t) \in \bar{\Omega}$ that
$$
\left\|\varphi^{i}\right\|_{\operatorname{var}}(\omega, t)=\left\|\varphi_{t}^{i}\right\|_{\operatorname{var}}(\omega)=\sup _{f \in U_{1}}\left(\varphi_{t}^{i}(f)\right)(\omega)=\sup _{k \in \mathbb{N}}\left(\varphi^{i}\left(u_{k}\right)\right)(\omega, t)
$$
As each $\varphi\left(u_{k}\right)$ is predictable, so is then $\|\varphi\|_{\text {var }}$ as a countable supremum.
\section{Stochastic integrals}
In this section, we introduce the stochastic integral of measure-valued processes $\varphi$ with respect to the semimartingale $S$. This will be achieved in three subsections. The first subsection defines this integral for elementary measure-valued processes $\varphi$. The second subsection presents some approximation results for general measure-valued processes $\varphi$, while the third (last) subsection  uses these approximation to deal with the integral of general measure-valued integrands $\varphi$. 
\subsection{Simple integrands and integrals}
To start our integration theory, let $\mathcal{E}$ be the family of $\mathbb{M}^{d}$-valued processes $\varphi$ of the form

\begin{equation*}
\varphi=\sum_{\ell=0}^{L-1} m_{\ell} I_{A_{\ell} \times\left(t_{\ell}, t_{\ell+1}\right]}=: \sum_{\ell=0}^{L-1} m_{\ell} I_{D_{\ell}} \tag{3.1}
\end{equation*}

with $L \in \mathbb{N}, 0 \leq t_{0}<t_{1}<\cdots<t_{L}<\infty$, each $m_{\ell} \in \mathbb{M}^{d}$ and each $A_{\ell} \in \mathcal{F}_{t_{\ell}}$ so that $D_{\ell} \in \mathcal{P}$. Note that the coefficients $m_{\ell}$ do not depend on $\omega$. For every $f \in C(\mathcal{K})$, the process

$$
\varphi(f)=\sum_{\ell=0}^{L-1} m_{\ell}(f) I_{D_{\ell}}=\sum_{\ell=0}^{L-1} m_{\ell}(f) I_{A_{\ell} \times\left(t_{\ell}, t_{\ell+1}\right]}
$$
is then an $\mathbb{R}^{d}$-valued very simple predictable process (i.e., in $\mathcal{H}$ ), which justifies calling $\varphi \in \mathcal{E}$ a very simple $\mathbb{M}^{d}$-valued weak* predictable process. Similarly, we call a process $N=\left(N_{t}\right)_{t \geq 0}$ a weak* semimartingale, or a weak process in $\mathcal{R}^{2}$, if the real-valued process $N(f)=\left(N_{t}(f)\right)_{t \geq 0}$ is a semimartingale, respectively in $\mathcal{R}^{2}$, for each $f \in C(\mathcal{K})$. Here, each $N_{t}(\omega)$ is assumed to be a linear (but not necessarily continuous) functional on $C(\mathcal{K})$, hence can be viewed as a (signed and finite) finitely additive measure on $\mathcal{B}(\mathcal{K})$, and then $N_{t}(f):=\int_{\mathcal{K}} f(z) N_{t}(\mathrm{~d} z)$. Such functionals are sometimes also called finite charges, and we sometimes call such a process $N$ a charge-valued process.

In order to define a stochastic integral $\varphi \bullet S$ for a fixed $\mathbb{R}^{d}$-valued semimartingale and suitably general $\mathbb{M}^{d}$-valued $\varphi$, we start with $\varphi \in \mathcal{E}$ and then extend to a larger class of integrands. Note that while both $\varphi$ and $S$ are $d$-dimensional, the resulting integral $\varphi \bullet S$ is one-dimensional, in analogy to the classic vector stochastic integral.

Because the description (2.2) of the semimartingale property is in prelocal form (i.e., on a right-open stochastic interval $\llbracket 0, \tau \llbracket$ ), we also construct the stochastic integral $\varphi \bullet S$ on $\llbracket 0, \tau \llbracket$ and only piece things together globally at the end.
\begin{lemma} Fix an $\mathbb{R}^{d}$-valued semimartingale $S$, a control process $V$ for $S$ and a stopping time $\tau$ with $V_{\tau-} \in L^{2}$. For each $\varphi \in \mathcal{E}$, there exists a process $\varphi \bullet S^{\tau-}=\left(\varphi \bullet S_{t}^{\tau-}\right)_{t \geq 0}$ with values in $\mathbb{M}$ which is a weak* semimartingale in $\mathcal{R}^{2}$. This process has the regular weak* Fubini property
\begin{equation*}
\left(\varphi \bullet S_{t}^{\tau-}\right)(f)=\varphi(f) \cdot S_{t}^{\tau-} \quad \text { for } t \geq 0 \text { and } f \in C(\mathcal{K}) \tag{3.2}
\end{equation*}
More precisely, (3.2) states that for each $f \in C(\mathcal{K})$, the two processes $\left(\varphi \bullet S^{\tau-}\right)(f)$ and $\varphi(f) \cdot S^{\tau-}$ are indistinguishable. Written out in terms of integrals, (3.2) takes the form
$$
\int_{\mathcal{K}} f(z)\left(\int_{0}^{t} \varphi_{r} \mathrm{~d} S_{r}^{\tau-}\right)(\mathrm{d} z)=\int_{0}^{t}\left(\int_{\mathcal{K}} f(z) \varphi_{r}(\mathrm{~d} z)\right) \mathrm{d} S_{r}^{\tau-} \quad \text { for } t \geq 0 \text { and } f \in C(\mathcal{K})
$$
which explains the terminology. In particular, we have for $\varphi \in \mathcal{E}$ the inequality
\begin{equation*}
\left\|\left(\varphi \bullet S^{\tau-}\right)(f)\right\|_{\mathcal{R}^{2}} \leq\|\varphi(f)\|_{L^{2}\left(\mu_{\tau, V}\right)} \quad \text { for } f \in C(\mathcal{K}) \tag{3.3}
\end{equation*}
Moreover, $\varphi \bullet S^{\tau-}$ also has the general weak* Fubini property
\begin{equation*}
\left(\varphi \bullet S_{t}^{\tau-}\right)(g)=\varphi(g) \cdot S_{t}^{\tau-} \quad \text { for } t \geq 0 \text { and } g \text { bounded and } \mathcal{B}(\mathcal{K}) \text {-measurable. } \tag{3.4}
\end{equation*}
Finally, (3.3) also holds with $f \in C(\mathcal{K})$ replaced by $g$ bounded and $\mathcal{B}(\mathcal{K})$-measurable.\end{lemma}
\begin{proof} For $\varphi=\sum_{\ell=0}^{L-1} m_{\ell} I_{D_{\ell}}=\sum_{\ell=0}^{L-1} m_{\ell} I_{A_{\ell} \times\left(t_{\ell}, t_{\ell+1}\right]}$, it is clear that we set
$$
\varphi \bullet S_{t}^{\tau-}:=\sum_{\ell=0}^{L-1} I_{A_{\ell}}\left(S_{t_{\ell+1} \wedge t}^{\tau-}-S_{t_{\ell} \wedge t}^{\tau-}\right)^{\top} m_{\ell} \quad \text { for } t \geq 0
$$
which clearly takes values in $\mathbb{M}$. Then we get for each $f \in C(\mathcal{K})$ that
\begin{equation*}
\left(\varphi \bullet S_{t}^{\tau-}\right)(f)=\sum_{\ell=0}^{L-1} I_{A_{\ell}} m_{\ell}(f)^{\top}\left(S_{t_{\ell+1} \wedge t}^{\tau-}-S_{t_{\ell} \wedge t}^{\tau-}\right)=\int_{0}^{t} H_{r} \mathrm{~d} S_{r}^{\tau-} \tag{3.5}
\end{equation*}
with the process
\begin{align*}
H_{r}(\omega):=\sum_{\ell=0}^{L-1} m_{\ell}(f) I_{D_{\ell}}(\omega, r) & =\sum_{\ell=0}^{L-1} \int_{\mathcal{K}} I_{A_{\ell} \times\left(t_{\ell}, t_{\ell+1}\right]}(\omega, r) f(z) m_{\ell}(\mathrm{d} z)  \tag{3.6}\\
& =\left(\int_{\mathcal{K}} f(z) \varphi_{r}(\mathrm{~d} z)\right)(\omega)=\varphi_{r}(f)(\omega) .
\end{align*}
Combining (3.5) with (3.6) gives (3.2). Moreover, the predictable process $H=\varphi(f)$ is bounded uniformly in $(\omega, r)$ by 
$$C:=\max \left\{\left\|m_{\ell}^{i}\right\|_{\text {var }}\|f\|_{\infty}: i=1, \ldots, d, \ell=0, \ldots, L-1\right\}$$
 and hence in $L^{2}\left(\mu_{\tau, V}\right)$ as $V_{\tau-} \in L^{2}$. Therefore $\left(\varphi \bullet S^{\tau-}\right)(f)$ is well defined as a semimartingale and in $\mathcal{R}^{2}$ by (2.4), which also directly gives (3.3). Finally, all operations in the above argument only involve finite sums, and so we can replace $f \in C(\mathcal{K})$ by any bounded $\mathcal{B}(\mathcal{K})$-measurable $g$ and obtain the same results. This completes the proof.\end{proof}
\begin{remark} 1) For elementary integrands $\varphi \in \mathcal{E}$, the regular weak* Fubini property (3.2) and the general one in (3.4) are equally easy to obtain. But for an extension to a larger class of integrands, it is simpler to start with (3.2) because the class $C(\mathcal{K})$ of test functions is separable, whereas the class of bounded $\mathcal{B}(\mathcal{K})$-measurable functions is not.\\
2) The above construction and setup as well as the subsequent extension to more general $\varphi$ are strongly inspired by the work of Björk et al. [8]. However, there is an important difference. In [8], the goal is to develop stochastic integration with respect to an integrator ( $\bar{S}$, say) taking values in $C([0, T])$; the natural integrands ( $\bar{\varphi}$, say) then take values in the dual space $C([0, T])^{*}={\mathbb{M}}([0, T])$, and the resulting integral $\int \bar{\varphi} \mathrm{d} \bar{S}$ has values in $\mathbb{R}$. (In the application in [8], $\bar{S}_{t}$ is the bond price curve at time $t$ of a term structure model, $\bar{\varphi}_{t}$ is a portfolio of bonds held at time $t$, and $\int \bar{\varphi} \mathrm{d} \bar{S}$ describes the cumulative gains/losses from trading over time via the dynamic strategy $\bar{\varphi}=\left(\bar{\varphi}_{t}\right)$.) In contrast, our integrator $S$ is $\mathbb{R}^{d}$-valued, and so using the analogous elementary $\mathbb{M}^{d}$-valued integrands $\varphi$ leads to an integral process $\varphi \bullet S$ which (at least for elementary $\varphi$ ) is again measure-valued (and charge-valued in general). As a consequence, the resulting space of integral processes has a more complicated structure, and in particular completeness, for a natural seminorm, is no longer clear.\end{remark}
\subsection{Approximation of general integrands}
To extend both $\varphi \mapsto \varphi \bullet S$ and the regular weak* Fubini property (3.2) from $\varphi \in \mathcal{E}$ to more general integrands, we use approximations and hence need appropriate seminorms on $\varphi$ and on real-valued weak* semimartingales $N$ in $\mathcal{R}^{2}$ like $\varphi \bullet S^{\tau-}$. We start on the side of the integrands. Throughout this section, we fix an $\mathbb{R}^{d}$-valued semimartingale $S$, a control process $V$ for $S$ and a stopping time $\tau$ with $V_{\tau-} \in L^{2}$, and denote by $\underline{u}=\left(u_{k}\right)_{k \in \mathbb{N}}$\\
a countable dense subset of the unit ball $U_{1}$ of $C(\mathcal{K})$ and by $\underline{\gamma}=\left(\gamma_{k}\right)_{k \in \mathbb{N}}$ a sequence in $(0, \infty)$ with $\sum_{k=1}^{\infty} \gamma_{k}=1$.
\begin{definition} For any $\mathbb{M}^{d}$-valued weak* predictable process $\varphi=\left(\varphi_{t}\right)_{t \geq 0}$, we define
$$
q_{\underline{\gamma}, \underline{u} ; \tau, V}(\varphi):=\left(\sum_{k=1}^{\infty} \gamma_{k}\left\|\varphi\left(u_{k}\right)\right\|_{L^{2}\left(\mu_{\tau, V}\right)}^{2}\right)^{1 / 2}
$$
and call $L^{2}(\underline{\gamma}, \underline{u} ; \tau, V)$ the space of all $\mathbb{M}^{d}$-valued weak* predictable $\varphi$ with $q_{\underline{\gamma}, \underline{u} ; \tau, V}(\varphi)<\infty$. We denote by $\Phi(\tau, V)$ the class of all $\mathbb{M}^{d}$-valued weak* predictable $\varphi$ satisfying
\begin{equation*}
C_{\varphi}:=\sup _{f \in U_{1} \backslash\{0\}} \frac{\|\varphi(f)\|_{L^{2}\left(\mu_{\tau, V}\right)}}{\|f\|_{\infty}}<\infty \tag{3.7}
\end{equation*}
For any $c>0$ and with $\mathbf{1}=(1, \ldots, 1) \in \mathbb{R}^{d}$, we also define the set $\mathcal{U}_{c}^{d} \subseteq \mathbb{M}^{d}$ by
$$
\mathcal{U}_{c}^{d}:=\left\{m \in \mathbb{M}^{d}:\|m\|_{\operatorname{var}} \leq c \mathbf{1}\right\}
$$
\end{definition}
\begin{remark}1) Any $\varphi \in \mathcal{E}$ has values in $\mathcal{U}_{c}^{d}$ with $c \leq \max _{i=1, \ldots, d, \ell=0,1, \ldots, L-1}\left\|m_{\ell}^{i}\right\|_{\text {var }}$. Next, if a weak* predictable $\varphi$ has values in $\bigcup_{c>0} \mathcal{U}_{c}^{d}$, the process $H:=\varphi(f)$ is predictable and $\mathbb{R}^{d}$-valued for any $f \in C(\mathcal{K})$ and satisfies $\left|H_{t}^{i}\right|=\left|\varphi_{t}^{i}(f)\right| \leq\left\|\varphi_{t}^{i}\right\|_{\text {var }}\|f\|_{\infty} \leq c\|f\|_{\infty}$, $i=1, \ldots, d$, for some $c>0$. This yields
\begin{equation*}
\|\varphi(f)\|_{L^{2}\left(\mu_{\tau, V}\right)}=\|H\|_{L^{2}\left(\mu_{\tau, V}\right)} \leq c\|f\|_{\infty} \mu_{\tau, V}(\bar{\Omega})=c\left\|V_{\tau-}\right\|_{L^{2}}\|f\|_{\infty} \tag{3.8}
\end{equation*}
so that $\varphi \in \Phi(\tau, V)$, with $C_{\varphi} \leq c\left\|V_{\tau-}\right\|_{L^{2}}$. Finally, any $\varphi \in \Phi(\tau, V)$ is clearly in $L^{2}(\underline{\gamma}, \underline{u} ; \tau, V)$ with $q_{\underline{\gamma}, \underline{u} ; \tau, V}(\varphi) \leq C_{\varphi}$. So we have
\begin{equation*}
\mathcal{E} \subseteq\left\{\left(\bigcup_{c>0} \mathcal{U}_{c}^{d}\right) \text {-valued weak* predictable } \varphi\right\} \subseteq \Phi(\tau, V) \subseteq L^{2}(\underline{\gamma}, \underline{u} ; \tau, V) \tag{3.9}
\end{equation*}
2) From (3.7), we see that an $\mathbb{M}^{d}$-valued weak $^{*}$ predictable process $\varphi$ is in $\Phi(\tau, V)$ if and only if $\varphi$ as a (clearly linear) mapping from $C(\mathcal{K})$ to the space of $\mathbb{R}^{d}$-valued predictable processes is continuous for $\|\cdot\|_{L^{2}\left(\mu_{\tau, V}\right)}$, meaning that we have $\|\varphi(f)\|_{L^{2}\left(\mu_{\tau, V}\right)} \leq c\|f\|_{\infty}$ for all $f \in C(\mathcal{K})$, for some $c>0$ which can depend on $\varphi$. Clearly, we can take $c=C_{\varphi}$.\\
3)  The space $L^{2}(\underline{\gamma}, \underline{u} ; \tau, V)$ as well as $q_{\underline{\gamma}, \underline{u} ; \tau, V}$ obviously depend on $\underline{\gamma}$ and $\underline{u}$ as well as on $V$ and $\tau$. But the above relations in 1) and 2) hold for any fixed choice of $\underline{\gamma}, \underline{u}$ and $V, \tau$.
\end{remark}
\begin{lemma} Under our standing assumptions in this subsection, we have:\\
1) $q_{\underline{\gamma}, \underline{u} ; \tau, V}$ is a seminorm on $L^{2}(\underline{\gamma}, \underline{u} ; \tau, V)$.\\
2)  For each $c>0$, the functional $q_{\underline{\gamma}, \underline{u} ; \tau, V}$ is sequentially continuous with respect to the weak* topology on the space of all weak predictable $\varphi$ with values in $\mathcal{U}_{c}^{d}$ : If a sequence $\left(\varphi^{n}\right)_{n \in \mathbb{N}}$ of such processes converges $\mu_{\tau, V}$-a.e. for the weak $k^{*}$ topology to some $\mathbb{M}^{d}$-valued weak* predictable $\varphi$, then also $\varphi$ takes values in $\mathcal{U}_{c}^{d}$ and $\lim _{n \rightarrow \infty} q_{\underline{\gamma}, \underline{u} ; \tau, V}\left(\varphi^{n}-\varphi\right)=0$.
\end{lemma}
\begin{proof} 1) Clearly $q_{\underline{\gamma}, \underline{u} ; \tau, V}$ is positively homogeneous. For the triangle inequality, we write

$$
q_{\underline{\gamma}, \underline{u} ; \tau, V}(\varphi)=\left(\sum_{k=1}^{\infty} \gamma_{k}\left\|\varphi\left(u_{k}\right)\right\|_{L^{2}\left(\mu_{\tau, V}\right)}^{2}\right)^{1 / 2}=\left(\sum_{k=1}^{\infty} \gamma_{k} x_{k}^{2}\right)^{1 / 2}=\|\underline{x}\|_{\ell^{2}(\gamma)}
$$
with $x_{k}:=\left\|\varphi\left(u_{k}\right)\right\|_{L^{2}\left(\mu_{\tau, V}\right)}$, and analogously for $\varphi^{\prime}$ with $\underline{x}^{\prime}$. So $q_{\underline{\gamma}, \underline{u} ; \tau, V}(\varphi)$ is just the $\ell^{2}$-norm of the sequence $\underline{x}=\left(x_{k}\right)_{k \in \mathbb{N}}$ for the probability measure on $\mathbb{N}$ having weights $\gamma_{k}$ for $k \in \mathbb{N}$. But $y_{k}:=\left\|\varphi\left(u_{k}\right)+\varphi^{\prime}\left(u_{k}\right)\right\|_{L^{2}\left(\mu_{\tau, V}\right)} \leq x_{k}+x_{k}^{\prime}$ by the triangle inequality in $L^{2}\left(\mu_{\tau, V}\right)$, and so
$$
\begin{aligned}
q_{\underline{\gamma}, \underline{u} ; \tau, V}\left(\varphi+\varphi^{\prime}\right) & =\|\underline{y}\|_{\ell^{2}(\gamma)} \leq\left(\sum_{k=1}^{\infty} \gamma_{k}\left(x_{k}+x_{k}^{\prime}\right)^{2}\right)^{1 / 2}=\left\|\underline{x}+\underline{x}^{\prime}\right\|_{\ell^{2}(\gamma)} \\
& \leq\|\underline{x}\|_{\ell^{2}(\gamma)}+\left\|\underline{x}^{\prime}\right\|_{\ell^{2}(\gamma)}=q_{\underline{\gamma}, \underline{u} ; \tau, V}(\varphi)+q_{\underline{\gamma}, \underline{u} ; \tau, V}\left(\varphi^{\prime}\right) .
\end{aligned}
$$
This proves the triangle inequality for $q_{\underline{\gamma}, \underline{\underline{u}} ; \tau, V}$ on $L^{2}(\underline{\gamma}, \underline{u} ; \tau, V)$.\\
2) Take a sequence $\left(\varphi^{n}\right)_{n \in \mathbb{N}}$ of $\mathbb{M}^{d}$-valued weak${ }^{*}$ predictable processes with $\left\|\varphi^{n}\right\|_{\text {var }} \leq c \mathbf{1}$ for all $n$ and suppose that $\left(\varphi^{n}\right)_{n \in \mathbb{N}}$ converges $\mu_{\tau, V}$-a.e. to some $\varphi$ for the weak* topology. Then $\varphi^{n}(f) \rightarrow \varphi(f) \mu_{\tau, V}$-a.e. for every $f \in C(\mathcal{K})$, and thus for $i=1, \ldots, d$,
\begin{equation*}
\left|\varphi_{t}^{i}(f)\right|=\lim _{n \rightarrow \infty}\left|\varphi_{t}^{n ; i}(f)\right| \leq \limsup _{n \rightarrow \infty}\left\|\varphi_{t}^{n ; i}\right\|_{\text {var }}\|f\|_{\infty} \leq c\|f\|_{\infty} \quad \mu_{\tau, V} \text {-a.e. } \tag{3.10}
\end{equation*}
Hence $\varphi$ also has values in $\mathcal{U}_{c}^{d}$ and is thus in $L^{2}(\underline{\gamma}, \underline{u} ; \tau, V)$ by (3.9). Replacing $\varphi^{n}$ by $\varphi^{n}-\varphi$ and $c$ by $2 c$, we can and do assume without loss of generality that $\varphi \equiv 0$. Then $\lim _{n \rightarrow \infty} \varphi^{n}\left(u_{k}\right)=0 \mu_{\tau, V}$-a.e. for all $k \in \mathbb{N}$, and using $\left\|\varphi^{n}\right\|_{\text {var }} \leq 2 c \mathbf{1}$ for all $n \in \mathbb{N}$ yields as in (3.10) that $\sup _{n \in \mathbb{N}}\left|\varphi^{n}\left(u_{k}\right)\right| \leq \sup _{n \in \mathbb{N}}\left|\left\|\varphi^{n}\right\|_{\operatorname{var}}\right|\left\|u_{k}\right\|_{\infty} \leq 2 c|\mathbf{1}| \in L^{2}\left(\mu_{\tau, V}\right)$ due to $V_{\tau-} \in L^{2}$. So dominated convergence gives $\lim _{n \rightarrow \infty} \varphi^{n}\left(u_{k}\right)=0$ in $L^{2}\left(\mu_{\tau, V}\right)$, for each $k \in \mathbb{N}$. Finally, (3.8) for $\varphi^{n}$ and $u_{k} \in U_{1}$ with $c$ replaced by $2 c$ yields $\left\|\varphi^{n}\left(u_{k}\right)\right\|_{L^{2}\left(\mu_{\tau, V}\right)}^{2} \leq 4 c^{2} E\left[V_{\tau-}^{2}\right]$. Now\\
first take $K_{0}$ large enough to obtain $4 c^{2} E\left[V_{\tau-}^{2}\right] \sum_{k=K_{0}+1}^{\infty} \gamma_{k} \leq \varepsilon^{2}$ and then use $\varphi^{n}\left(u_{k}\right) \rightarrow 0$ in $L^{2}\left(\mu_{\tau, V}\right)$ as $n \rightarrow \infty$, for each $k$, to take $n$ large enough to get $\gamma_{k}\left\|\varphi^{n}\left(u_{k}\right)\right\|_{L^{2}\left(\mu_{\tau, V}\right)}^{2} \leq \varepsilon^{2} / K_{0}$ for $k=1, \ldots, K_{0}$. Then the definition of $q_{\underline{\gamma}, \underline{u} ; \tau, V}$ and the above estimates imply that

$$
\left(q_{\underline{\gamma}, \underline{u} ; \tau, V}\left(\varphi^{n}\right)\right)^{2}=\sum_{k=1}^{\infty} \gamma_{k}\left\|\varphi^{n}\left(u_{k}\right)\right\|_{L^{2}\left(\mu_{\tau, V}\right)}^{2} \leq 2 \varepsilon^{2}
$$
This proves the assertion.\end{proof}

The next approximation result is crucial for our construction of stochastic integrals of measure-valued processes.

\begin{lemma} Under our standing assumptions in this subsection, the set $\mathcal{E}$ is dense in $L^{2}(\underline{\gamma}, \underline{u} ; \tau, V)$ with respect to the seminorm $q_{\underline{\gamma}, \underline{u} ; \tau, V}$.\end{lemma}

\begin{proof} This is similar to the very concisely written result in Bj\"ork et al. [8, Lemma 2.3], but for completeness and readability, we give the proof in detail. First of all, $\mathcal{E} \subseteq L^{2}(\underline{\gamma}, \underline{u} ; \tau, V)$ as seen in (3.9). If we consider very simple $\mathcal{P}$-measurable $\varphi$ as in (3.1) where we only have sets $D_{\ell} \in \mathcal{P}$ instead of predictable rectangles $D_{\ell}=A_{\ell} \times\left(t_{\ell}, t_{\ell+1}\right]$, these $\varphi$ are still in $L^{2}(\underline{\gamma}, \underline{u} ; \tau, V)$ if they are valued in $\mathcal{U}_{c}^{d}$ for some $c>0$, again by (3.9).\\
To prove denseness, we start with $\varphi \in L^{2}(\underline{\gamma}, \underline{u} ; \tau, V)$. Then $\underline{\gamma} \subseteq(0, \infty)$ and
$$
q_{\underline{\gamma}, \underline{u} ; \tau, V}(\varphi)=\left(\sum_{k=1}^{\infty} \gamma_{k}\left\|\varphi\left(u_{k}\right)\right\|_{L^{2}\left(\mu_{\tau, V}\right)}^{2}\right)^{1 / 2}<\infty
$$
implies that $\varphi\left(u_{k}\right) \in L^{2}\left(\mu_{\tau, V}\right)$ for all $k \in \mathbb{N}$. For an arbitrary $c>0$, the process
$$
\varphi^{(c)}:=\varphi I_{\left\{\|\varphi\|_{\mathrm{var}} \leq c \mathbf{1}\right\}} \quad \text { (defined coordinatewise) }
$$
is weak$^{*}$ predictable like $\varphi$ because $\|\varphi\|_{\text {var }}$ is predictable, and $\varphi-\varphi^{(c)}=\varphi I_{\left\{\|\varphi\|_{\text {var }}>c 1\right\}}$ (again coordinatewise). This implies that $\lim _{c \rightarrow \infty}\left(\varphi-\varphi^{(c)}\right)(f)=0 \mu_{\tau, V}$-a.e. for every $f \in C(\mathcal{K})$ because $\varphi$ takes values in $\mathbb{M}^{d}$, and $\left|\left(\varphi-\varphi^{(c)}\right)\left(u_{k}\right)\right| \leq\left|\varphi\left(u_{k}\right)\right| \in L^{2}\left(\mu_{\tau, V}\right)$ for each $k$. Dominated convergence hence gives $\left(\varphi-\varphi^{(c)}\right)\left(u_{k}\right) \rightarrow 0$ in $L^{2}\left(\mu_{\tau, V}\right)$ as $c \rightarrow \infty$, for each $k$. Next, $q_{\underline{\gamma}, \underline{u} ; \tau, V}(\varphi)<\infty$ allows us to again use dominated convergence, for the measure on $\mathbb{N}$ with weights $\left(\gamma_{k}\right)_{k \in \mathbb{N}}$, and obtain $q_{\underline{\gamma}, \underline{u} ; \tau, V}\left(\varphi-\varphi^{(c)}\right) \rightarrow 0$ as $c \rightarrow \infty$. Each $\varphi^{(c)}$\\
takes values in $\mathcal{U}_{c}^{d}=\left\{m \in \mathbb{M}^{d}:\|m\|_{\text {var }} \leq c \mathbf{1}\right\}$, and so it is enough to show that $\mathcal{E}$ is dense with respect to $q_{\underline{\gamma}, \underline{\underline{u}} ; \tau, V}$ in the set of all weak* predictable processes $\varphi$ valued in $\mathcal{U}_{c}^{d}$.

Now by the Banach-Alaoglu theorem, each ball $\mathcal{U}_{c}^{d}$ is compact in the weak* topology on $\mathbb{M}^{d}=\left((C(\mathcal{K}))^{d}\right)^{*}$, and metrisable because $C(\mathcal{K})$ is separable; see Section 2.3. Hence $\mathcal{U}_{c}^{d}$ is separable for the weak* topology and we can choose a countable weak* dense subset $\left(m_{j}\right)_{j \in \mathbb{N}}$ in $\mathcal{U}_{c}^{d}$. Define $\mathcal{M}_{n}:=\left\{m_{1}, m_{2}, \ldots, m_{n}\right\}$ for all $n \in \mathbb{N}$ and note that because $\underline{u}=\left(u_{k}\right)_{k \in \mathbb{N}}$ is dense in $U_{1}$, the metric $\delta\left(m, m^{\prime}\right):=\sum_{k=1}^{\infty} 2^{-k}\left|m\left(u_{k}\right)-m^{\prime}\left(u_{k}\right)\right|$ on $\mathcal{U}_{c}^{d}$ induces the weak* topology on any $\mathcal{U}_{c}^{d}$. If $\varphi$ is a weak* predictable process valued in $\mathcal{U}_{c}^{d}, \delta(\varphi, m)$ is predictable for every fixed $m \in \mathcal{U}_{c}^{d}$, and so is then the process $\delta\left(\varphi, \mathcal{M}_{n}\right)=\min _{i=1, \ldots, n} \delta\left(\varphi, m_{i}\right)$. As $\left(m_{j}\right)_{j \in \mathbb{N}}$ is weak $^{*}$ dense in $\mathcal{U}_{c}^{d}$, we know that $\delta\left(\varphi, \mathcal{M}_{n}\right) \searrow 0 \mu_{\tau, V}$-a.e. as $n \rightarrow \infty$, and the pairwise disjoint sets
$$
C_{j}^{n}:=\left\{\delta\left(\varphi, \mathcal{M}_{n}\right)=\delta\left(\varphi, m_{j}\right) \text { and } \delta\left(\varphi, m_{i}\right)>\delta\left(\varphi, m_{j}\right) \text { for } i<j\right\} \subseteq \bar{\Omega}
$$
where the distance of $\varphi$ to $\mathcal{M}_{n}$ is realized for the first time in $m_{j} \in \mathcal{M}_{n}$ are in $\mathcal{P}$. The process $\Psi^{n}:=\sum_{j=1}^{n} m_{j} I_{C_{j}^{n}}$ is thus a very simple $\mathcal{P}$-measurable $\mathcal{U}_{c}^{d}$-valued process, and
$$
\delta\left(\varphi, \Psi^{n}\right)=\sum_{j=1}^{n} I_{C_{j}^{n}} \delta\left(\varphi, m_{j}\right)=\delta\left(\varphi, \mathcal{M}_{n}\right) \longrightarrow 0 \quad \mu_{\tau, V} \text {-a.e. }
$$
shows that $\Psi^{n} \rightarrow \varphi$ in the weak* topology which is metrised by $\delta$. This proves that because the weak* predictable process $\varphi$ takes values in $\mathcal{U}_{c}^{d}$, it is actually strongly $\mathcal{P}$-measurable in the sense (usual for Banach-valued random processes) that it is an a.e. pointwise limit of elementary $\mathcal{P}$-measurable processes. As a consequence, Lemma 3.5, 2) implies that $\left(\Psi^{n}\right)_{n \in \mathbb{N}}$ converges to $\varphi$ with respect to $q_{\underline{\gamma}, \underline{u} ; \tau, V}$.

The above sequence $\left(\Psi^{n}\right)$ is not yet in $\mathcal{E}$ because the sets $C_{j}^{n}$ are in $\mathcal{P}$, but not necessarily predictable rectangles of the form $A_{s} \times(s, t]$ with $s \leq t$ and $A_{s} \in \mathcal{F}_{s}$. But each $\Psi^{n}$ has the form $\Psi^{n}=\sum_{j=1}^{J(n)} m_{j} I_{C_{j}^{n}}$ with $m_{j} \in \mathcal{U}_{c}^{d}$ and $C_{j}^{n} \in \mathcal{P}$ pairwise disjoint. This is a finite linear combination, and so it is enough to consider $\Psi=m I_{C}$ with $m \in \mathcal{U}_{c}^{d}, C \in \mathcal{P}$ and show that this can be approximated with respect to $q_{\underline{\gamma}, \underline{u} ; \tau, V}$ by a sequence $\left(\varphi^{i}\right)_{i \in \mathbb{N}}$ in $\mathcal{E}$. For this, note that the predictable $\sigma$-field $\mathcal{P}$ is generated by the sets of the form $D_{\ell}=A_{\ell} \times\left(t_{\ell}, t_{\ell+1}\right]$ with $A_{\ell} \in \mathcal{F}_{t_{\ell}}$, so that the real-valued predictable processes of the form $H=\sum_{\ell=0}^{L-1} c_{\ell} I_{A_{\ell} \times\left(t_{\ell}, t_{\ell+1}\right]}$\\
with $L \in \mathbb{N}$ and $c_{\ell} \in \mathbb{R}$ are dense in $L^{2}\left(\mu_{\tau, V}\right)$. (This uses that the measure $\mu_{\tau, V}$ is finite as $V_{\tau-} \in L^{2}$.) So we can choose a sequence $\left(H^{i}\right)_{i \in \mathbb{N}}$ of such processes with $H^{i} \rightarrow I_{C}$ in $L^{2}\left(\mu_{\tau, V}\right)$, and by passing to a subsequence, still called $\left(H^{i}\right)$, we can achieve that $H^{i} \rightarrow I_{C}$ $\mu_{\tau, V^{-}}$-a.e. Moreover, as the limit $I_{C}$ is bounded by 1 , we can even modify the $H^{i}$ so that they are bounded by $1+\varepsilon$, uniformly in $\omega, t, i$. If we now set $\varphi^{i}:=m H^{i}$, each $\varphi^{i}$ takes values in $\mathcal{U}_{c(1+\varepsilon)}^{d}$ because $m \in \mathcal{U}_{c}^{d}$, and recalling that $\Psi=m I_{C}$, we get $\mu_{\tau, V^{-}}$a.e. that
$$
\left|\varphi^{i}\left(u_{k}\right)-\Psi\left(u_{k}\right)\right|=\left|m\left(u_{k}\right)\right|\left|H^{i}-I_{C}\right| \leq\left|\|m\|_{\operatorname{var}}\right|\left\|u_{k}\right\|_{\infty}\left|H^{i}-I_{C}\right| \longrightarrow 0
$$
simultaneously for all and uniformly in $k \in \mathbb{N}$; note that $\left\|u_{k}\right\|_{\infty} \leq 1$. Therefore
$$
\delta\left(\varphi^{i}, m I_{C}\right)=\sum_{k=1}^{\infty} 2^{-k}\left|\varphi^{i}\left(u_{k}\right)-\Psi\left(u_{k}\right)\right| \leq\left|\|m\|_{\mathrm{var}}\right|\left|H^{i}-I_{C}\right| \longrightarrow 0 \quad \mu_{\tau, V^{-}} \text {a.e., }
$$
and we have already seen above that $\delta$ metrises the weak* topology on $\mathcal{U}_{c(1+\varepsilon)}^{d}$. So we get $\varphi^{i} \rightarrow m I_{C} \mu_{\tau, V^{-}}$-a.e. for the weak* topology as $i \rightarrow \infty$, and Lemma 3.5,2) then yields that $q_{\underline{\gamma}, \underline{u} ; \tau, V}\left(\varphi^{i}-m I_{C}\right) \rightarrow 0$ as $i \rightarrow \infty$.\\
Combining all the above results implies the assertion and completes the proof.\end{proof}

\noindent By Remark 3.4,2), a weak* predictable process $\varphi$ is in $\Phi(\tau, V) \subseteq L^{2}(\underline{\gamma}, \underline{u} ; \tau, V)$ if and only if it is continuous as a mapping from $C(\mathcal{K})$ to $L^{2}\left(\mu_{\tau, V}\right)$. The next result sharpens the approximation obtained in Lemma 3.6 - it shows that for $\varphi \in \Phi(\tau, V)$, one can find an approximating sequence $\left(\varphi^{n}\right)_{n \in \mathbb{N}}$ with the same type of continuity as $\varphi$, uniformly in $n$.

\begin{corollary} Under our standing assumptions in this subsection, for every $\varphi \in \Phi(\tau, V)$, there exists a sequence $\left(\varphi^{n}\right)_{n \in \mathbb{N}} \subseteq \mathcal{E}$ with $\lim _{n \rightarrow \infty} q_{\underline{\gamma}, \underline{u} ; \tau, V}\left(\varphi^{n}-\varphi\right)=0$ and
\begin{equation*}
\sup _{n \in \mathbb{N}}\left\|\varphi^{n}(f)\right\|_{L^{2}\left(\mu_{\tau, V}\right)} \leq c\|f\|_{\infty} \quad \text { for all } f \in C(\mathcal{K}) \tag{3.11}
\end{equation*}
with some constant $c \in(0, \infty)$.\end{corollary}
\begin{proof} We go back to the proof of Lemma 3.6 and examine in more detail how the approximating sequence is constructed there. Start with $\varphi \in \Phi(\tau, V)$, fix $b>1$, take $c:=b-1$ and set $\varphi^{(c)}:=\varphi I_{\left\{\|\varphi\|_{\mathrm{var}} \leq c 1\right\}}$. Approximate, with respect to $q_{\underline{\gamma}, \underline{u} ; \tau, V}$, first $\varphi^{(c)}$ by a sequence\\
$\left(\Psi^{n}\right)_{n \in \mathbb{N}}$ of the form $\Psi^{n}=\sum_{j=1}^{J(n)} m_{j} I_{C_{j}^{n}}$, and then each $\Psi^{n}$ by some $\varphi^{n} \in \mathcal{E}$. Note (from the proof of Lemma 3.6) that $\Psi^{n}$ takes values in $\mathcal{U}_{c}^{d}$ and $\varphi^{n}$ hence in $\mathcal{U}_{c(1+\varepsilon)}^{d} \subseteq \mathcal{U}_{b}^{d}$ for $\varepsilon \leq \frac{1}{b-1}$ because then $c(1+\varepsilon) \leq b$. This gives $\lim _{n \rightarrow \infty} q_{\underline{\gamma}, \underline{u} ; \tau, V}\left(\varphi^{n}-\varphi\right)=0$, and for any $f \in C(\mathcal{K})$,
$$
\left|\varphi^{n}(f)\right| \leq\left|\varphi^{n}(f)-\Psi^{n}(f)\right|+\left|\Psi^{n}(f)-\varphi^{(c)}(f)\right|+\left|\varphi^{(c)}(f)-\varphi(f)\right|+|\varphi(f)| .
$$
Because $\varphi^{n}$ has values in $\mathcal{U}_{b}^{d}$ and $\Psi^{n}$ in $\mathcal{U}_{c}^{d} \subseteq \mathcal{U}_{b}^{d}$, we get $\left|\varphi^{n}(f)-\Psi^{n}(f)\right| \leq|b \mathbf{1}|\|f\|_{\infty}$, and in the same way, $\left|\Psi^{n}(f)-\varphi^{(c)}(f)\right| \leq|c \mathbf{1}|\|f\|_{\infty} \leq|b \mathbf{1}|\|f\|_{\infty}$. Moreover, by the definition of $\varphi^{(c)}$, we have $\left|\varphi^{(c)}(f)-\varphi(f)\right| \leq|\varphi(f)|$, and so we obtain from (3.8) and $\varphi \in \Phi(\tau, V)$ that
$$
\left\|\varphi^{n}(f)\right\|_{L^{2}\left(\mu_{\tau, V}\right)} \leq 2 b\left\|V_{\tau-}\right\|_{L^{2}}\|f\|_{\infty}+2\|\varphi(f)\|_{L^{2}\left(\mu_{\tau, V}\right)} \leq\left(2 b\left\|V_{\tau-}\right\|_{L^{2}}+2 C_{\varphi}\right)\|f\|_{\infty} .
$$
As the constant on the right-hand side does not depend on $n$, this completes the proof.\end{proof}
\subsection{Limits of measure-valued semimartingales}
In the last section, we have constructed a stochastic integral process $\varphi \bullet S^{\tau-}$ for measure-valued integrands $\varphi \in \mathcal{E}$. This yields a measure-valued weak* semimartingale in $\mathcal{R}^{2}$, and we have also seen that any $\varphi \in L^{2}(\underline{\gamma}, \underline{u} ; \tau, V)$ can be approximated by a sequence $\left(\varphi^{n}\right)_{n \in \mathbb{N}} \subseteq \mathcal{E}$. It seems natural to try and define $\varphi \bullet S^{\tau-}$ for more general $\varphi$ than in $\mathcal{E}$ as a limit, in a suitable sense, of $\left(\varphi^{n} \bullet S^{\tau-}\right)_{n \in \mathbb{N}}$, and so we first need a concept of convergence.

As in Section 3.2, throughout this subsection, we fix an $\mathbb{R}^{d}$-valued semimartingale $S$, a control process $V$ for $S$ and a stopping time $\tau$ with $V_{\tau-} \in L^{2}$, and we denote by $\underline{u}=\left(u_{k}\right)_{k \in \mathbb{N}}$ a countable dense subset of the unit ball $U_{1}$ of $C(\mathcal{K})$ and by $\underline{\gamma}=\left(\gamma_{k}\right)_{k \in \mathbb{N}}$ a sequence in $(0, \infty)$ with $\sum_{k=1}^{\infty} \gamma_{k}=1$.
\begin{definition}Definition 3.8. For a weak* process $N=\left(N_{t}\right)_{t \geq 0}$ in $\mathcal{R}^{2}$, we define
$$
r_{\underline{\gamma}, \underline{u}}(N):=\left(\sum_{k=1}^{\infty} \gamma_{k}\left\|N\left(u_{k}\right)\right\|_{\mathcal{R}^{2}}^{2}\right)^{1 / 2}
$$
and call $\mathcal{N}^{2}(\underline{\gamma}, \underline{u})$ the space of all weak* processes $N$ in $\mathcal{R}^{2}$ with $r_{\underline{\gamma}, \underline{u}}(N)<\infty$.\end{definition}
\begin{lemma} Under our standing assumptions in this subsection, we have:\\
1) $r_{\underline{\gamma}, \underline{u}}$ is a seminorm on $\mathcal{N}^{2}(\underline{\gamma}, \underline{u})$.\\
2) For any $\varphi \in \mathcal{E}$, we have the inequality
\begin{equation*}
r_{\underline{\gamma}, \underline{u}}\left(\varphi \bullet S^{\tau-}\right) \leq q_{\underline{\gamma}, \underline{u} ; \tau, V}(\varphi) . \tag{3.12}
\end{equation*}
\end{lemma}
\begin{proof} 1) This is completely analogous to the proof of 1) in Lemma 3.5, just replacing $x_{k}=\left\|\varphi\left(u_{k}\right)\right\|_{L^{2}\left(\mu_{\tau, V}\right)}$ by $x_{k}:=\left\|N\left(u_{k}\right)\right\|_{\mathcal{R}^{2}}$.\\
2) The inequality (3.3) in Lemma 3.1 gives $\left\|\left(\varphi \bullet S^{\tau-}\right)\left(u_{k}\right)\right\|_{\mathcal{R}^{2}} \leq\left\|\varphi\left(u_{k}\right)\right\|_{L^{2}\left(\mu_{\tau, V}\right)}$ for each $\varphi \in \mathcal{E}$ and all $k \in \mathbb{N}$. So (3.12) follows from the definitions of $r_{\underline{\gamma}, \underline{u}}$ and $q_{\underline{\gamma}, \underline{u} ; \tau, V}$.
\end{proof}

The next result is crucial for taking limits in $\mathcal{N}^{2}(\underline{\gamma}, \underline{u})$.
\begin{lemma} Under our standing assumptions in this subsection, suppose that the sequence $\left(N^{n}\right)_{n \in \mathbb{N}} \subseteq \mathcal{N}^{2}(\underline{\gamma}, \underline{u})$ is a Cauchy sequence for $r_{\underline{\gamma}, \underline{u}}$. If it also satisfies
\begin{equation*}
\sup _{n \in \mathbb{N}}\left\|N^{n}(f)\right\|_{\mathcal{R}^{2}} \leq c\|f\|_{\infty} \quad \text { for all } f \in C(\mathcal{K}) \tag{3.13}
\end{equation*}
with a constant $c \in(0, \infty)$ (that can depend on the sequence), then $\left(N^{n}\right)_{n \in \mathbb{N}}$ converges with respect to $r_{\underline{\gamma}, \underline{u}}$ to some $N^{\infty} \in \mathcal{N}^{2}(\underline{\gamma}, \underline{u})$ which also satisfies
\begin{equation*}
\left\|N^{\infty}(f)\right\|_{\mathcal{R}^{2}} \leq c\|f\|_{\infty} \quad \text { for all } f \in C(\mathcal{K}) \tag{3.14}
\end{equation*}
\end{lemma}
\begin{proof} From $\underline{\gamma} \subseteq(0, \infty)$ and the definition of $r_{\underline{\gamma}, \underline{u}}$, we see that the Cauchy property of $\left(N^{n}\right)_{n \in \mathbb{N}}$ for $r_{\underline{\gamma}, \underline{u}}$ implies for each $k \in \mathbb{N}$ that $\left(N^{n}\left(u_{k}\right)\right)_{n \in \mathbb{N}}$ is a Cauchy sequence in $\mathcal{R}^{2}$ and therefore has a limit $N^{\infty, k}$ in $\mathcal{R}^{2}$. We define $N^{\infty}$ on $\underline{u} \subseteq C(\mathcal{K})$ by
$$
N_{t}^{\infty}\left(u_{k}\right):=N_{t}^{\infty, k} \quad \text { for all } t \geq 0 \text { and } k \in \mathbb{N}
$$
and get that $N^{\infty}\left(u_{k}\right)=N^{\infty, k}$ is in $\mathcal{R}^{2}$ for each $k \in \mathbb{N}$. We now claim that the definition of $N^{\infty}$ can be extended from $\underline{u}$ to all of $C(\mathcal{K})$ in a linear way. To see this, take any $f \in U_{1}$ (which is enough by using linearity) and set
\begin{equation*}
N^{\infty}(f):=\lim _{\ell \rightarrow \infty} N^{\infty}\left(u_{\ell}^{\prime}\right) \quad \text { in } \mathcal{R}^{2} \tag{3.15}
\end{equation*}
for any sequence $\left(u_{\ell}^{\prime}\right)_{\ell \in \mathbb{N}} \subseteq \underline{u}$ which converges to $f$ in $C(\mathcal{K})$. Such a sequence exists because $\underline{u}$ is dense in $U_{1}$, and so it remains to show that the limit in (3.15) exists and does not depend on the approximating sequence $\left(u_{\ell}^{\prime}\right)_{\ell \in \mathbb{N}}$.\\
Because $u_{\ell}^{\prime} \rightarrow f$ in $C(\mathcal{K}),\left(u_{\ell}^{\prime}\right)_{\ell \in \mathbb{N}}$ is Cauchy in $C(\mathcal{K})$. Now write
\begin{equation*}
N^{\infty}\left(u_{\ell}^{\prime}\right)-N^{\infty}\left(u_{j}^{\prime}\right)=\left(N^{\infty}\left(u_{\ell}^{\prime}\right)-N^{n}\left(u_{\ell}^{\prime}\right)\right)+\left(N^{n}\left(u_{\ell}^{\prime}\right)-N^{n}\left(u_{j}^{\prime}\right)\right)+\left(N^{n}\left(u_{j}^{\prime}\right)-N^{\infty}\left(u_{j}^{\prime}\right)\right) \tag{3.16}
\end{equation*}
and note that the first and third differences converge to 0 in $\mathcal{R}^{2}$ as $n \rightarrow \infty$ by the construction of $N^{\infty}$. The second difference can be estimated in $\mathcal{R}^{2}$ by
$$
\sup _{n \in \mathbb{N}}\left\|N^{n}\left(u_{\ell}^{\prime}\right)-N^{n}\left(u_{j}^{\prime}\right)\right\|_{\mathcal{R}^{2}}=\sup _{n \in \mathbb{N}}\left\|N^{n}\left(u_{\ell}^{\prime}-u_{j}^{\prime}\right)\right\|_{\mathcal{R}^{2}} \leq c\left\|u_{\ell}^{\prime}-u_{j}^{\prime}\right\|_{\infty}
$$
thanks to the linearity of each $N^{n}$ and (3.13), and so $\left(N^{\infty}\left(u_{\ell}^{\prime}\right)\right)_{\ell \in \mathbb{N}}$ is a Cauchy sequence in $\mathcal{R}^{2}$ and hence convergent in $\mathcal{R}^{2}$. Thus the limit in (3.15) exists, and we claim that it does not depend on the chosen approximating sequence $\left(u_{\ell}^{\prime}\right)_{\ell \in \mathbb{N}}$ for $f$. Indeed, if $\left(\tilde{u}_{\ell}\right)_{\ell \in \mathbb{N}}$ is another approximating sequence for $f$ with limit $\tilde{N}^{\infty}(f):=\lim _{\ell \rightarrow \infty} N^{\infty}\left(\tilde{u}_{\ell}\right)$, we can write
$$
\tilde{N}^{\infty}(f)-N^{\infty}(f)=\left(\tilde{N}^{\infty}(f)-N^{\infty}\left(\tilde{u}_{\ell}\right)\right)+\left(N^{\infty}\left(\tilde{u}_{\ell}\right)-N^{\infty}\left(u_{\ell}^{\prime}\right)\right)+\left(N^{\infty}\left(u_{\ell}^{\prime}\right)-N^{\infty}(f)\right) .
$$
Then the first and third differences converge to 0 in $\mathcal{R}^{2}$ as $\ell \rightarrow \infty$ by the definitions of $\tilde{N}^{\infty}$ and $N^{\infty}$, respectively, and the second difference converges to 0 in $\mathcal{R}^{2}$ as $\ell \rightarrow \infty$ by the same argument as for the Cauchy property in (3.16); note that $\left\|\tilde{u}_{\ell}-u_{\ell}^{\prime}\right\|_{\infty} \rightarrow 0$ as $\ell \rightarrow \infty$ because both $\left(\tilde{u}_{\ell}\right)_{\ell \in \mathbb{N}}$ and $\left(u_{\ell}^{\prime}\right)_{\ell \in \mathbb{N}}$ converge to $f$ in $C(\mathcal{K})$. So we obtain $\tilde{N}^{\infty}(f)=N^{\infty}(f)$.

By the preceding arguments, $N^{\infty}$ is well defined on all of $C(\mathcal{K})$, and linear and a weak* process in $\mathcal{R}^{2}$, both by construction. To show that the sequence $\left(N^{n}\right)_{n \in \mathbb{N}}$ converges to $N^{\infty}$ for $r_{\underline{\gamma}, \underline{u}}$, we first note that $N^{n}\left(u_{k}\right) \rightarrow N^{\infty}\left(u_{k}\right)$ in $\mathcal{R}^{2}$ as $n \rightarrow \infty$, for each $k \in \mathbb{N}$. For fixed $\varepsilon>0$ and each finite $K$, we thus obtain for large $n$ that
$$
\begin{aligned}
\sum_{k=1}^{K} \gamma_{k}\left\|N^{n}\left(u_{k}\right)-N^{\infty}\left(u_{k}\right)\right\|_{\mathcal{R}^{2}}^{2} & =\lim _{m \rightarrow \infty} \sum_{k=1}^{K} \gamma_{k}\left\|N^{n}\left(u_{k}\right)-N^{m}\left(u_{k}\right)\right\|_{\mathcal{R}^{2}}^{2} \\
& \leq \limsup _{m \rightarrow \infty} \sum_{k=1}^{\infty} \gamma_{k}\left\|N^{n}\left(u_{k}\right)-N^{m}\left(u_{k}\right)\right\|_{\mathcal{R}^{2}}^{2} \\
& =\limsup _{m \rightarrow \infty}\left(r_{\underline{\gamma}, \underline{u}}\left(N^{n}-N^{m}\right)\right)^{2} \leq \varepsilon^{2}
\end{aligned}
$$
independently of $K$, by the Cauchy property of $\left(N^{n}\right)_{n \in \mathbb{N}}$ for $r_{\underline{\gamma}, \underline{u}}$. This shows by letting $K \rightarrow \infty$ that $r_{\underline{\gamma}, \underline{u}}\left(N^{n}-N^{\infty}\right) \leq \varepsilon$ for $n$ large enough. So $\left(N^{n}\right)_{n \in \mathbb{N}}$ converges to $N^{\infty}$ with respect to $r_{\underline{\gamma}, \underline{u}}$ and $N^{\infty}$ lies in $\mathcal{N}^{2}(\underline{\gamma}, \underline{u})$.

Finally, to prove (3.14), we fix $f \in C(\mathcal{K})$ and take a sequence $\left(u_{\ell}^{\prime}\right)_{\ell \in \mathbb{N}} \subseteq \underline{u}$ converging to $f$ in $C(\mathcal{K})$. Then
$$
\begin{aligned}
\left\|N^{\infty}(f)\right\|_{\mathcal{R}^{2}} \leq & \left\|N^{\infty}(f)-N^{\infty}\left(u_{\ell}^{\prime}\right)\right\|_{\mathcal{R}^{2}}+\left\|N^{\infty}\left(u_{\ell}^{\prime}\right)-N^{n}\left(u_{\ell}^{\prime}\right)\right\|_{\mathcal{R}^{2}} \\
& +\sup _{n \in \mathbb{N}}\left\|N^{n}\left(u_{\ell}^{\prime}\right)-N^{n}(f)\right\|_{\mathcal{R}^{2}}+\sup _{n \in \mathbb{N}}\left\|N^{n}(f)\right\|_{\mathcal{R}^{2}}
\end{aligned}
$$
and the fourth summand is at most $c\|f\|_{\infty}$ by (3.13). The other summands all converge to 0 as $\ell \rightarrow \infty$ and $n \rightarrow \infty$ - the first by the definition (3.15) of $N^{\infty}(f)$, the second as in (3.16) by the construction of $N^{\infty}$, and the third by (3.13) because $u_{\ell}^{\prime} \rightarrow f$ in $C(\mathcal{K})$. This shows (3.14) and completes the proof.\end{proof}
\begin{remark} 1) By the construction in the proof of Lemma 3.10, $N^{\infty}: C(\mathcal{K}) \rightarrow \mathcal{R}^{2}$ is continuous, meaning that $f^{n} \rightarrow f$ in $C(\mathcal{K})$ implies that $N^{\infty}\left(f^{n}\right) \rightarrow N^{\infty}(f)$ in $\mathcal{R}^{2}$. However, this is different from continuity of $N_{t}^{\infty}(\omega): C(\mathcal{K}) \rightarrow \mathbb{R}$. In other words, the proof of Lemma 3.10 does not show (and it is not true in general) that the limit process $N^{\infty}$ takes values in $\mathbb{M}$. This is the reason why the stochastic integral $\varphi \bullet S$ we construct below is not always a measure-valued process, but only charge-valued in general. We give a counterexample later in Section 7.\\
2) Because we construct it pointwise on $C(\mathcal{K})$ as a limit in $\mathcal{R}^{2}$, the process $N^{\infty}$ need not be a weak* semimartingale in general even if the sequence $\left(N^{n}\right)_{n \in \mathbb{N}}$ consists of weak* semimartingales. This is because the subspace of semimartingales is not closed in $\mathcal{R}^{2}$. But see ${ }^{* * *}[$ check what can be said here] below for some positive results.\\
3)  As in Remark 3.4, 3), the space $\mathcal{N}^{2}(\underline{\gamma}, \underline{u})$ as well as $r_{\underline{\gamma}, \underline{u}}$ depend on $\underline{\gamma}$ and $\underline{u}$. But the above results again hold for any fixed choice of $\underline{\gamma}$ and $\underline{u}$.
\end{remark}
\section{Two new stochastic Fubini theorems}
In this section, we combine the preceding results to construct a stochastic integral process $\varphi \bullet S$ for fairly general measure-valued integrands $\varphi$, and we use this to prove new stochastic Fubini theorems. The idea to get $\varphi \bullet S$ is simple. We first prelocalise, approximate $\varphi$ on a stochastic interval $\llbracket 0, \tau \llbracket$ by a sequence $\left(\varphi^{n}\right)_{n \in \mathbb{N}}$ of elementary integrands in $\mathcal{E}$ and define $\varphi \bullet S^{\tau-}$ to be the limit of the sequence $\left(\varphi^{n} \bullet S^{\tau-}\right)_{n \in \mathbb{N}}$. While the approximation in $\mathcal{E}$ can be done for any $\varphi \in L^{2}(\underline{\gamma}, \underline{u} ; \tau, V)$, the existence of the limit of $\left(\varphi^{n} \bullet S^{\tau-}\right)_{n \in \mathbb{N}}$ needs the stronger assumption that $\varphi \in \Phi(\tau, V)$. In the end, we paste things together.\\
Let $S=\left(S_{t}\right)_{t \geq 0}$ be a fixed $\mathbb{R}^{d}$-valued semimartingale and recall from Section 2.2 the family $\mathcal{V}(S)$ of control processes $V$ for $S$.
\begin{definition} We denote by $\Phi$ the family of all $\mathbb{M}^{d}$-valued weak* predictable processes $\varphi$ such that for some control process $V \in \mathcal{V}(S)$,
\begin{equation*}
\mathcal{D}\left(\|\varphi\|_{\text {var }} ; V\right)=\int\|\| \varphi_{r} \|\left._{\text {var }}\right|^{2} \mathrm{~d} V_{r} \text { is a finite-valued process (i.e., in } \mathcal{R}^{0} \text { ). } \tag{4.1}
\end{equation*}
\end{definition}
\begin{remark} As seen at the end of Section 2.2, (4.1) is equivalent to saying that the predictable $\mathbb{R}^{d}$-valued process $\|\varphi\|_{\text {var }}=\left(\left\|\varphi_{t}\right\|_{\text {var }}\right)_{t \geq 0}$ is integrable with respect to $S$.

For any $\varphi \in \Phi$, the processes $V$ and $\mathcal{D}\left(\|\varphi\|_{\text {var }} ; V\right)$ are both in $\mathcal{R}^{0}$. So they are prelocally bounded and we can find a localising sequence $\left(\tau_{M}\right)_{M \in \mathbb{N}}$ of stopping times with $V_{\tau_{M}-} \in L^{2}$ and $\mathcal{D}_{\tau_{M}-}\left(\|\varphi\|_{\text {var }} ; V\right) \in L^{2}$ for each $M \in \mathbb{N}$. Thus $\mu_{\tau_{M}, V}$ is a finite measure and we have
$$
E\left[V_{\tau_{M}-} \int_{0}^{\tau_{M}-}\left|\left\|\varphi_{r}\right\|_{\mathrm{var}}\right|^{2} \mathrm{~d} V_{r}\right]<\infty
$$
which means that $\|\varphi\|_{\text {var }} \in L^{2}\left(\mu_{\tau_{M}, V}\right)$. The estimate $\left|\varphi_{r}(f)\right| \leq\|f\|_{\infty}\left|\left\|\varphi_{r}\right\|_{\text {var }}\right|$ for any $f \in C(\mathcal{K})$ then implies that
$$
\sup _{f \in U_{1} \backslash\{0\}} \frac{\|\varphi(f)\|_{L^{2}\left(\mu_{\tau_{M}, V}\right)}}{\|f\|_{\infty}} \leq\|\| \varphi\left\|_{\mathrm{var}}\right\|_{L^{2}\left(\mu_{\tau_{M}, V}\right)}<\infty
$$
so that $\varphi \in \Phi\left(\tau_{M}, V\right)$ for every $M \in \mathbb{N}$; see Definition 3.3. In summary, any $\mu_{\tau_{M}, V}$ is a finite measure and any $\varphi \in \Phi$ is in each $\Phi\left(\tau_{M}, V\right)$, for a suitable localizing sequence $\left(\tau_{M}\right)_{M \in \mathbb{N}}$.\end{remark}

With these preparations, we are now ready for our first main result. Our first new stochastic Fubini theorem is called regular because its test functions $f$ are continuous.
\begin{theorem} Let $S=\left(S_{t}\right)_{t \geq 0}$ be an $\mathbb{R}^{d}$-valued semimartingale. For every measure-valued integrand $\varphi \in \Phi$, there exists a stochastic integral process $\varphi \bullet S=\left(\varphi \bullet S_{t}\right)_{t \geq 0}$ which is a weak* process in $\mathcal{R}^{0}$ (and hence prelocally in $\mathcal{R}^{2}$ ). It is a linear and continuous mapping from $C(\mathcal{K})$ to $\mathcal{R}^{0}$ and satisfies the regular weak* Fubini property that for all $f \in C(\mathcal{K})$,
\begin{equation*}
(\varphi \bullet S)(f)=\varphi(f) \cdot S \quad \text { up to indistinguishability on } \bar{\Omega}=\Omega \times(0, \infty) \text {. } \tag{4.2}
\end{equation*}
Written in integral form, this means that
\begin{equation*}
\int_{\mathcal{K}} f(z)\left(\int_{0}^{t} \varphi_{r} \mathrm{~d} S_{r}\right)(\mathrm{d} z)=\int_{0}^{t}\left(\int_{\mathcal{K}} f(z) \varphi_{r}(\mathrm{~d} z)\right) \mathrm{d} S_{r} \quad \text { for all } t \geq 0, P-\text { a.s. } \tag{4.3}
\end{equation*}
The process $\varphi \bullet S$ is uniquely determined on $C(\mathcal{K})$ by (4.2).\end{theorem}
\begin{proof} Fix a countable dense subset $\underline{u}=\left(u_{k}\right)_{k \in \mathbb{N}}$ of the unit ball $U_{1}$ in $C(\mathcal{K})$ and a sequence $\underline{\gamma}=\left(\gamma_{k}\right)_{k \in \mathbb{N}}$ in $(0, \infty)$ with $\sum_{k=1}^{\infty} \gamma_{k}=1$. Take a control process $V \in \mathcal{V}(S)$ for $S$ and a localising sequence $\left(\tau_{M}\right)_{M \in \mathbb{N}}$ of stopping times with $V_{\tau_{M-}} \in L^{2}$ and $\mathcal{D}_{\tau_{M-}}\left(\|\varphi\|_{\text {var }} ; V\right) \in L^{2}$ for each $M \in \mathbb{N}$. Fix $M \in \mathbb{N}$ and write $\tau:=\tau_{M}$ for brevity. Then $\varphi \in \Phi$ is in $\Phi(\tau, V)$ and there exists by Lemma 3.6 a sequence $\left(\varphi^{n}\right)_{n \in \mathbb{N}} \subseteq \mathcal{E}$ converging to $\varphi$ for $q_{\underline{\gamma}, \underline{u} ; \tau, V}$. For any $n \in \mathbb{N}$, the measure-valued stochastic integral $\varphi^{n} \bullet S^{\tau-}$ is therefore by Lemma 3.1 well defined, a (measure-valued) weak* semimartingale in $\mathcal{R}^{2}$, and in $\mathcal{N}^{2}(\underline{\gamma}, \underline{u})$ because $r_{\underline{\gamma}, \underline{u}}\left(\varphi^{n} \bullet S^{\tau-}\right) \leq q_{\underline{\gamma}, \underline{u} ; \tau, V}\left(\varphi^{n}\right)$ by (3.12) in Lemma 3.9. The same inequality shows that the sequence $\left(N^{n}\right)_{n \in \mathbb{N}}:=\left(\varphi^{n} \bullet S^{\tau-}\right)_{n \in \mathbb{N}}$ is Cauchy for $r_{\underline{\gamma}, \underline{u}}$ because $\left(\varphi^{n}\right)_{n \in \mathbb{N}}$ is Cauchy for $q_{\underline{\gamma}, \underline{u} ; \tau, V}$ as it converges (to $\varphi$ ). Moreover, (3.3) in Lemma 3.1 gives the inequality
\begin{equation*}
\left\|\left(\varphi^{n} \bullet S^{\tau-}\right)(f)\right\|_{\mathcal{R}^{2}} \leq\left\|\varphi^{n}(f)\right\|_{L^{2}\left(\mu_{\tau, V}\right)} \quad \text { for all } f \in C(\mathcal{K}) \tag{4.4}
\end{equation*}
Because $\varphi \in \Phi(\tau, V)$, Corollary 3.7 allows us to choose the sequence $\left(\varphi^{n}\right)_{n \in \mathbb{N}}$ in such a way that it satisfies (3.11), i.e., $\sup _{n \in \mathbb{N}}\left\|\varphi^{n}(f)\right\|_{L^{2}\left(\mu_{\tau, V}\right)} \leq c\|f\|_{\infty}$ for all $f \in C(\mathcal{K})$ with some constant $c \in(0, \infty)$. Combining this with (4.4) gives
$$
\sup _{n \in \mathbb{N}}\left\|N^{n}(f)\right\|_{\mathcal{R}^{2}} \leq \sup _{n \in \mathbb{N}}\left\|\varphi^{n}(f)\right\|_{L^{2}\left(\mu_{\tau, V}\right)} \leq c\|f\|_{\infty} \quad \text { for all } f \in C(\mathcal{K}) \text {, }
$$
which allows us to apply Lemma 3.10 and conclude that $\left(N^{n}\right)_{n \in \mathbb{N}}$ converges with respect to $r_{\underline{\gamma}, \underline{u}}$ to some $N^{\infty} \in \mathcal{N}^{2}(\underline{\gamma}, \underline{u})$. We set $\varphi \bullet S^{\tau-}:=N^{\infty}$ and note that this is linear and continuous as a map from $C(\mathcal{K})$ to $\mathcal{R}^{2}$ by construction; see Remark 3.11, 1). Moreover, due to (3.14) in Lemma 3.10, we also have
\begin{equation*}
\left\|\left(\varphi \bullet S^{\tau-}\right)(f)\right\|_{\mathcal{R}^{2}} \leq c\|f\|_{\infty} \quad \text { for all } f \in C(\mathcal{K}) \tag{4.5}
\end{equation*}
To prove the regular weak* Fubini property (4.2), we first note that this holds for each $\varphi^{n} \in \mathcal{E}$ due to (3.2) in Lemma 3.1. By linearity, it is enough to consider $f \in U_{1}$, and then we can take a sequence $\left(u_{\ell}^{\prime}\right)_{\ell \in \mathbb{N}} \subseteq \underline{u}$ with $\lim _{\ell \rightarrow \infty} u_{\ell}^{\prime}=f$ in $C(\mathcal{K})$. Now we write
$$
\begin{aligned}
( & \left.\bullet S^{\tau-}\right)(f)-\varphi(f) \cdot S^{\tau-} \\
= & \left(\left(\varphi \bullet S^{\tau-}\right)(f)-\left(\varphi \bullet S^{\tau-}\right)\left(u_{\ell}^{\prime}\right)\right)+\left(\left(\varphi \bullet S^{\tau-}\right)\left(u_{\ell}^{\prime}\right)-\left(\varphi^{n} \bullet S^{\tau-}\right)\left(u_{\ell}^{\prime}\right)\right) \\
& +\left(\left(\varphi^{n} \bullet S^{\tau-}\right)\left(u_{\ell}^{\prime}\right)-\varphi^{n}\left(u_{\ell}^{\prime}\right) \cdot S^{\tau-}\right) \\
& +\left(\varphi^{n}\left(u_{\ell}^{\prime}\right) \cdot S^{\tau-}-\varphi\left(u_{\ell}^{\prime}\right) \cdot S^{\tau-}\right)+\left(\varphi\left(u_{\ell}^{\prime}\right) \cdot S^{\tau-}-\varphi(f) \cdot S^{\tau-}\right)
\end{aligned}
$$
and consider the five differences on the right-hand side one by one. The third vanishes by the regular weak* Fubini property (3.2) for $\varphi^{n} \in \mathcal{E}$. Next, $\varphi^{n} \rightarrow \varphi$ for $q_{\gamma, \underline{u} ; \tau, V}$ implies by the definition of $q_{\underline{\gamma}, \underline{u} ; \tau, V}$ that $\varphi^{n}\left(u_{k}\right) \rightarrow \varphi\left(u_{k}\right)$ in $L^{2}\left(\mu_{\tau, V}\right)$ for each $u_{k} \in \underline{u}$. So $\varphi^{n}\left(u_{\ell}^{\prime}\right) \rightarrow \varphi\left(u_{\ell}^{\prime}\right)$ in $L^{2}\left(\mu_{\tau, V}\right)$ as $n \rightarrow \infty$ and thus the fourth difference tends to 0 in $\mathcal{R}^{2}$ as $n \rightarrow \infty$ by (2.4), for each $\ell \in \mathbb{N}$. Next, $\varphi \in \Phi(\tau, V)$ gives $\left\|\varphi\left(u_{\ell}^{\prime}\right)-\varphi(f)\right\|_{L^{2}\left(\mu_{\tau, V}\right)} \leq C_{\varphi}\left\|u_{\ell}^{\prime}-f\right\|_{\infty}$ (see Definition 3.3), and $u_{\ell}^{\prime} \rightarrow f$ in $C(\mathcal{K})$ by construction. So $\varphi\left(u_{\ell}^{\prime}\right) \rightarrow \varphi(f)$ in $L^{2}\left(\mu_{\tau, V}\right)$, and the same argument as just above implies that the fifth difference tends to 0 in $\mathcal{R}^{2}$ as $\ell \rightarrow \infty$.

Now by construction, $\varphi \bullet S^{\tau-}$ is the limit of the sequence $\left(\varphi^{n} \bullet S^{\tau-}\right)_{n \in \mathbb{N}}$ for $r_{\underline{\gamma}, \underline{u}}$, and we have $\left(\varphi \bullet S^{\tau-}\right)(f)=\lim _{\ell \rightarrow \infty}\left(\varphi \bullet S^{\tau-}\right)\left(u_{\ell}^{\prime}\right)$ in $\mathcal{R}^{2}$ by the definition (3.15) of $\varphi \bullet S^{\tau-}=N^{\infty}$ from Lemma 3.10. In consequence, the first difference tends to 0 in $\mathcal{R}^{2}$ as $\ell \rightarrow \infty$. Finally, the definition of $r_{\underline{\gamma}, \underline{\underline{u}}}$ yields $\left(\varphi^{n} \bullet S^{\tau-}\right)\left(u_{k}\right) \rightarrow\left(\varphi \bullet S^{\tau-}\right)\left(u_{k}\right)$ in $\mathcal{R}^{2}$ as $n \rightarrow \infty$ for each $u_{k} \in \underline{u}$, and so also the second difference tends to 0 in $\mathcal{R}^{2}$ as $n \rightarrow \infty$, for each $\ell \in \mathbb{N}$. By taking first $\ell$ large enough and then $n$ large enough for fixed $\ell$, this implies that $\left(\varphi \bullet S^{\tau-}\right)(f)$ and $\varphi(f) \cdot S^{\tau-}$ are indistinguishable, for each $f \in C(\mathcal{K})$, and proves (4.2).

It is clear that $\varphi \bullet S^{\tau-}$ is uniquely determined on $C(\mathcal{K})$ by (4.2) when restricting this to $\llbracket 0, \tau \llbracket$, and like the right-hand side of (4.2), $\varphi \bullet S^{\tau-}$ does not depend on the choice of the control process $V \in \mathcal{V}(S)$. Moreover, it is also independent of the choice of $\underline{u}$ and $\underline{\gamma}$. To see this, take another countable dense subset $\underline{u}^{\prime}$ of $U_{1}$ and another sequence $\underline{\gamma}^{\prime}$ in $(0, \infty)$ summing to 1 . Denote by $\underline{u} \cup \underline{u}^{\prime}$ and $\underline{\gamma} \cup \underline{\gamma}^{\prime}$ the set and sequence resulting from interlacing $\underline{u}, \underline{u}^{\prime}$ and $\underline{\gamma}, \underline{\gamma}^{\prime}$, respectively, by putting the primed quantities at even and the unprimed ones at odd indices. Then obviously $\left(r_{\underline{\gamma} \cup \underline{\gamma}^{\prime},\underline{u} \cup \underline{u}^{\prime}}(\cdot)\right)^2=\left(r_{\underline{\gamma},\underline{u}}(\cdot)\right)^2+\left(r_{\underline{\gamma}^{\prime},\underline{u}^{\prime}}(\cdot)\right)^2$ and  $\left(q_{\underline{\gamma} \cup \underline{q}^{\prime}, \underline{\underline{u}} \cup \underline{u}^{\prime} ; V}(\cdot)\right)^2=\left(q_{\underline{\gamma},\underline{u} ;\tau, V}(\cdot)\right)^2+\left(q_{\underline{\gamma}^{\prime},\underline{u}^{\prime} ; V}(\cdot)\right)^2$, and so an approximation with respect to $q_{\underline{\gamma} \cup \underline{q}^{\prime}, \underline{\underline{u}} \cup \underline{u}^{\prime} ; V}$ automatically implies one with respect to both $q_{\underline{\gamma}, \underline{u} ; \tau, V}$ and $q_{\underline{\underline{q}}^{\prime}, \underline{\underline{u}} ; V}$. By uniqueness, the process $\varphi \bullet S^{\tau-}$ constructed in $\mathcal{N}^{2}\left(\underline{\gamma} \cup \underline{\gamma}^{\prime}, \underline{u} \cup \underline{u}^{\prime}\right)$ therefore coincides with\\
both the processes constructed in $\mathcal{N}^{2}(\underline{\gamma}, \underline{u})$ and $\mathcal{N}^{2}\left(\underline{\gamma}^{\prime}, \underline{u}^{\prime}\right)$.

Now recall that $\tau$ is shorthand for $\tau_{M}$ with a fixed $M \in \mathbb{N}$. So $\varphi \bullet S^{\tau_{M}-}$ is defined and unique for each $M \in \mathbb{N}$, which gives a consistent definition of $\varphi \bullet S$ on each $\llbracket 0, \tau_{M} \llbracket$. As $\tau_{M} \nearrow \infty$, the process $\varphi \bullet S$ is therefore well defined and has the stochastic Fubini property on all of $\bar{\Omega}$. Moreover, $\varphi \bullet S$ is a continuous mapping on each $\llbracket 0, \tau_{M} \llbracket$ from $C(\mathcal{K})$ to $\mathcal{R}^{2} \subseteq \mathcal{R}^{0}$ and hence maps $C(\mathcal{K})$ into $\mathcal{R}^{0}$. For fixed $t \geq 0$ and $f \in C(\mathcal{K})$, combining $((\varphi \bullet S)(f))_{t}^{*} \leq\left(\left(\varphi \bullet S^{\tau_{M}-}\right)(f)\right)_{\infty}^{*} I_{\left\{t<\tau_{M}\right\}}+((\varphi \bullet S)(f))_{t}^{*} I_{\left\{t \geq \tau_{M}\right\}}$ with Markov's inequality and (4.5) gives
$$
\begin{aligned}
P\left[((\varphi \bullet S)(f))_{t}^{*}>2 \varepsilon\right] & \leq P\left[\left(\left(\varphi \bullet S^{\tau_{M}-}\right)(f)\right)_{\infty}^{*}>\varepsilon\right]+P\left[\tau_{M} \leq t\right] \\
& \leq \frac{1}{\varepsilon^{2}}\left\|\left(\varphi \bullet S^{\tau_{M}-}\right)(f)\right\|_{\mathcal{R}^{2}}^{2}+P\left[\tau_{M} \leq t\right] \\
& \leq \frac{1}{\varepsilon^{2}} c^{2}\|f\|_{\infty}^{2}+P\left[\tau_{M} \leq t\right] .
\end{aligned}
$$
As $M \rightarrow \infty$, the second summand goes to 0 because $\tau_{M} \nearrow \infty$, and so $\varphi \bullet S: C(\mathcal{K}) \rightarrow \mathcal{R}^{0}$ is continuous for the ucp-topology. This completes the proof..\end{proof}
In some situations, one needs a stochastic Fubini theorem for less regular test functions than $f \in C(\mathcal{K})$. The next result goes in that direction.
\begin{theorem} Let $S=\left(S_{t}\right)_{t \geq 0}$ be an $\mathbb{R}^{d}$-valued semimartingale. For every measure-valued integrand $\varphi \in \Phi$, the stochastic integral process $\varphi \bullet S=\left(\varphi \bullet S_{t}\right)_{t \geq 0}$ from Theorem 4.3 also satisfies the general weak* Fubini property that for all $g \in b B_{1}(\mathcal{K})$,
\begin{equation*}
(\varphi \bullet S)(g)=\varphi(g) \cdot S \quad \text { up to indistinguishability on } \bar{\Omega}=\Omega \times(0, \infty) \text {. } \tag{4.6}
\end{equation*}
This implies in particular that $(\varphi \bullet S)(g)$ is in $\mathcal{R}^{0}$. Moreover, for every closed subset $D$ of $\mathcal{K}$, we have $(\varphi \bullet S)(D)=\varphi(D) \cdot S$, i.e.,
\begin{equation*}
\int_{D}\left(\int_{0}^{t} \varphi_{r} \mathrm{~d} S_{r}\right)(\mathrm{d} z)=\int_{0}^{t}\left(\int_{D} \varphi_{r}(\mathrm{~d} z)\right) \mathrm{d} S_{r} \quad \text { for all } t \geq 0, P \text {-a.s. } \tag{4.7}
\end{equation*}
\end{theorem}
\begin{proof} Because $\mathcal{K}$ is compact and Hausdorff, Urysohn's lemma implies that for every closed $D \subseteq \mathcal{K}$, we can write $I_{D}=\lim _{n \rightarrow \infty} f_{n}$ pointwise for a sequence $\left(f_{n}\right)_{n \in \mathbb{N}}$ in $U_{1} \subseteq C(\mathcal{K})$.\\
Hence $I_{D}$ is in $b B_{1}(\mathcal{K})$ and (4.7) follows from (4.6). Moreover, the right-hand side of (4.6) is in $\mathcal{R}^{0}$ if it is well defined, and so is then the left-hand side.\\
The proof of (4.6) is similar to the proof of Theorem 4.3 in that we first prelocalise and argue on stochastic intervals $\llbracket 0, \tau \llbracket$. Fix $\varphi \in \Phi$ and a control process $V$ for $S$, and let $\left(\tau_{M}\right)_{M \in \mathbb{N}}$ be a localizing sequence of stopping times such that for every $M \in \mathbb{N}$,
\begin{equation*}
V_{\tau_{M}-} \in L^{2}, \quad \mathcal{D}_{\tau_{M}-}\left(\|\varphi\|_{\text {var }} ; V\right) \in L^{2} \tag{4.8}
\end{equation*}
Fix $M \in \mathbb{N}$ and write $\tau:=\tau_{M}$ for brevity. By the proof of Theorem 4.3, $\varphi \bullet S^{\tau-}$ is well defined and a linear mapping from $C(\mathcal{K})$ to $\mathcal{R}^{2}$, and so each $\varphi \bullet S_{t}^{\tau-}(\omega)$ can be identified with a finite charge on $\mathcal{B}(\mathcal{K})$. As a consequence, $\left(\varphi \bullet S^{\tau-}\right)(g)$ is a well-defined process, even for any Borel-measurable bounded function $g$ on $\mathcal{K}$.\\
Now fix $g \in b B_{1}(\mathcal{K})$ and write $g=\lim _{n \rightarrow \infty} f_{n}$ pointwise for a sequence $\left(f_{n}\right)_{n \in \mathbb{N}} \subseteq C(\mathcal{K})$. As $g$ is bounded by $c$, say, we may assume that $\left\|f_{n}\right\|_{\infty} \leq 2 c$ for all $n$, and so the sequence $h_{n}:=g-f_{n}, n \in \mathbb{N}$, tends to 0 pointwise on $\mathcal{K}$ and is bounded by $4 c$. We first argue that $\varphi(g)$ is predictable and in $L^{2}\left(\mu_{\tau, V}\right)$, so that $\varphi(g) \cdot S^{\tau-}$ is well defined and in $\mathcal{R}^{2}$ due to (2.4). By (4.8), we have $\left|\|\varphi\|_{\text {var }}\right|<\infty \mu_{\tau, V}$-a.e., and so dominated convergence yields $\varphi\left(f_{n}\right) \rightarrow \varphi(g) \mu_{\tau, V^{-}}$a.e. as $n \rightarrow \infty$. Fatou's lemma for $\mu_{\tau, V}$ together with $\varphi \in \Phi(\tau, V)$ yields via (3.7) that
$$
\|\varphi(g)\|_{L^{2}\left(\mu_{\tau, V}\right)} \leq \liminf _{n \rightarrow \infty}\left\|\varphi\left(f_{n}\right)\right\|_{L^{2}\left(\mu_{\tau, V}\right)} \leq C_{\varphi} \sup _{n \in \mathbb{N}}\left\|f_{n}\right\|_{\infty}<\infty
$$
and so $\varphi(g)$ satisfies the required integrability for being in $L^{2}\left(\mu_{\tau, V}\right)$. To check predictability (up to indistinguishability), we note that $\left|\varphi\left(f_{n}\right)\right| \leq\left\|f_{n}\right\|_{\infty}\left|\|\varphi\|_{\text {var }}\right|$ leads via (4.8) to
$$
\mathcal{D}_{\tau-}\left(\varphi\left(f_{n}\right) ; V\right)=\int_{0}^{\tau-}\left|\varphi_{r}\left(f_{n}\right)\right|^{2} \mathrm{~d} V_{r} \leq\left\|f_{n}\right\|_{\infty}^{2} \mathcal{D}_{\tau-}\left(\|\varphi\|_{\text {var }} ; V\right)<\infty \quad P \text {-a.s. }
$$
so that dominated convergence first gives $\mathcal{D}_{\tau-}\left(\varphi\left(f_{n}\right) ; V\right) \rightarrow \mathcal{D}_{\tau-}(\varphi(g) ; V) P$-a.s. Thanks to (4.8), we can again use dominated convergence to obtain $\varphi(g)=\lim _{n \rightarrow \infty} \varphi\left(f_{n}\right)$ in $L^{2}\left(\mu_{\tau, V}\right)$ and hence also $\mu_{\tau, V}$-a.e. along a subsequence. As each $\varphi\left(f_{n}\right)$ is predictable, so is $\varphi(g)$ (up to indistinguishability). Thus $\varphi(g) \cdot S^{\tau-}$ is well defined and in $\mathcal{R}^{2}$, and so is then $\left(\varphi \bullet S^{\tau-}\right)(g)$ if we can prove (4.6) on $\llbracket 0, \tau \llbracket$.\\
To achieve that last point, we write
$$
\begin{aligned}
\left(\varphi \bullet S^{\tau-}\right)(g)-\varphi(g) \cdot S^{\tau-}= & \left(\left(\varphi \bullet S^{\tau-}\right)(g)-\left(\varphi \bullet S^{\tau-}\right)\left(f_{n}\right)\right) \\
& +\left(\left(\varphi \bullet S^{\tau-}\right)\left(f_{n}\right)-\varphi\left(f_{n}\right) \cdot S^{\tau-}\right) \\
& +\left(\varphi\left(f_{n}\right) \cdot S^{\tau-}-\varphi(g) \cdot S^{\tau-}\right) .
\end{aligned}
$$
The second summand on the right-hand side is 0 by the regular Fubini property (4.2). We have just seen that $\varphi\left(f_{n}\right) \rightarrow \varphi(g)$ in $L^{2}\left(\mu_{\tau, V}\right)$ so that the third summand converges to 0 in $\mathcal{R}^{2}$ by (2.4). For the first summand, we have $h_{n}=g-f_{n} \rightarrow 0$ pointwise and $\left|h_{n}\right| \leq 4 c$ for all $n$. As each $\varphi \bullet S_{t}^{\tau-}(\omega)$ can be identified with a finite charge on $\mathcal{B}(\mathcal{K})$, we get $\left(\varphi \bullet S^{\tau-}\right)\left(f_{n}\right) \rightarrow\left(\varphi \bullet S^{\tau-}\right)(g)$ for $\mu_{\tau, V^{-}}$almost all $(\omega, t)$ by dominated convergence; see Bhaskara Rao/Bhaskara Rao [6, Theorem 4.6.14]. This completes the proof.\end{proof}
\begin{remark} Apart from being more general then Theorem 4.3, Theorem 4.4 also has the advantage that it allows us to replace $\varphi$ by $I_{D} \varphi$ for a closed set $D \subseteq \mathcal{K}$ and hence obtain (4.3) with an integral over $D$ instead of over $\mathcal{K}$. This cannot be done by simply replacing $f$ by $I_{D} f$ in Theorem 4.3 as $I_{D} f$ is no longer in $C(\mathcal{K})$. That extra generality will be used in Section 6 below.\end{remark}
\section{An application to Volterra semimartingales}
Let $S=\left(S_{t}\right)_{t \geq 0}$ be an $\mathbb{R}^{d}$-valued semimartingale. In this section, we study stochastic processes $X=\left(X_{t}\right)_{t \geq 0}$ of the form
\begin{equation*}
X_{t}=\int_{0}^{t} \Psi_{t, s} \mathrm{~d} S_{s}, \quad t \geq 0 \tag{5.1}
\end{equation*}
for a two-parameter process $\Psi=\left(\Psi_{t, s}\right)_{t \geq 0,0 \leq s \leq t}$ with suitable measurability and integrability properties (made precise later). We look for conditions on $\Psi$ which ensure that $X$ is again a semimartingale, and we also want to say more about its decomposition in that case.\\
Processes of the form (5.1) are often called Volterra(-type) processes, and there is a large and growing body of literature around them. Forward or backward stochastic Volterra integral equations look for a process $X$ solving (in the forward case, say) an equation of the form $X_{t}=X_{0}+\int_{0}^{t} a\left(t, s,\left(X_{r}\right)_{r \leq s}\right) \mathrm{d} s+\int_{0}^{t} b\left(t, s,\left(X_{r}\right)_{r \leq s}\right) \mathrm{d} W_{s} ;$ see e.g. Hernández [14] for a recent\\
article with many more references. More specifically, stochastic Volterra equations take the form 
$$X_{t}=x_{0}(t)+\int_{0}^{t} K_{\mu}(s, t) \mu\left(s, X_{s}\right) \mathrm{d} s+\int_{0}^{t} K_{\sigma}(s, t) \sigma\left(s, X_{s}\right) \mathrm{d} W_{s},$$
 and one then studies existence, uniqueness and properties of solutions; see for instance Prömel/Scheffels [27]. If one specifies further to 
$$X_{t}=X_{0}+\int_{0}^{t} K(t-s) \mu\left(X_{s}\right) \mathrm{d} s+\int_{0}^{t} K(t-s) \sigma\left(X_{s}\right) \mathrm{d} W_{s}$$
 with $\mu, \sigma \sigma^{\top}$ affine, there is a large literature on such affine Volterra processes; see e.g. Abi Jaber et al. [1] for an important contribution.\\
In a different direction, processes of the form $X_{t}=\int_{0}^{t} K(t, s) \mathrm{d} S_{s}$ with $S$ a semimartingale, or $S=L$ a Lévy process, or $S=W$ a Brownian motion, have been studied for path properties, stochastic calculus, support properties and other topics; see e.g. Neuman [26], Bender et al. [5], Di Nunno et al. [11], Kalinin [17], Baudoin/Nualart [4], to name just a few (diverse, but not necessarily representative) contributions. A key difference between the above two directions is whether $X$ is given endogenously as the solution of some equation or exogenously from given quantities $\Psi, S$ not depending on $X$.\\
The question whether $X$ in (5.1) is a semimartingale seems to have two main lines of approach. If $X_{t}=\int_{0}^{t} g(t-s) \mathrm{d} L_{s}$, with a deterministic function $g$ and a Lévy process $L$, is a convolution-type integral or moving average process, Basse/Pedersen [3] have given necessary and sufficient condition, on $g$ and the Lévy triplet of $L$, for $X$ to be a semimartingale. An extension to $X_{t}^{\vartheta}=\int_{0}^{t} K^{\vartheta}(t, s) \mathrm{d} W_{s}$ with a Brownian motion $W$ and $K^{\vartheta}(t, s)$ of a particular form has recently been presented by El Omari [12]. At the other end of the scale is the work by Protter [28] who assumes that for all $s$ and $\omega$, (5.2) $t \mapsto \Psi_{t, s}(\omega)=\Psi(t, s, \omega)$ is in $C^{1}$ with $t \mapsto \psi(t, s, \omega):=\frac{\partial \Psi}{\partial t}(t, s, \omega)$ locally Lipschitz. He then writes $X$ as
\begin{equation*}
X_{t}=\int_{0}^{t} \Psi_{t, s} \mathrm{~d} S_{s}=\int_{0}^{t} \Psi_{s, s} \mathrm{~d} S_{s}+\int_{0}^{t}\left(\Psi_{t, s}-\Psi_{s, s}\right) \mathrm{d} S_{s} \tag{5.3}
\end{equation*}
uses (5.2) to get
\begin{equation*}
\Psi_{t, s}(\omega)-\Psi_{s, s}(\omega)=\int_{s}^{t} \psi(r, s, \omega) \mathrm{d} r \tag{5.4}
\end{equation*}
and handles the last term in (5.3) by writing
\begin{equation*}
\int_{0}^{t}\left(\Psi_{t, s}-\Psi_{s, s}\right) \mathrm{d} S_{s}=\int_{0}^{t}\left(\int_{s}^{t} \psi(r, s, \cdot) \mathrm{d} r\right) \mathrm{d} S_{s}=\int_{0}^{t}\left(\int_{0}^{r} \psi(r, s, \cdot) \mathrm{d} S_{s}\right) \mathrm{d} r . \tag{5.5}
\end{equation*}
Note that (5.2) implies via (5.4) that all the $t \mapsto \Psi_{t, s}$ are absolutely continuous with respect to one fixed (namely Lebesgue) measure, and so the last step in (5.5) follows by a standard stochastic Fubini theorem. Thus the last term in (5.3) is absolutely continuous with respect to Lebesgue measure and hence of finite variation and a semimartingale.

As we discuss later in Section 6, the assumption (5.2) or (5.4) means that we are in the so-called dominated case. Let us now see what we can do if we relax this restrictive condition. So let $\Psi=\left(\Psi_{t, s}\right)_{t \geq 0,0 \leq s \leq t}$ be an $\mathbb{R}^{d}$-valued two-parameter process such that for each $t \geq 0$, the process $\left(\Psi_{t, s}\right)_{0 \leq s \leq t}$ is predictable and $S$-integrable on $[0, t]$, and such that the diagonal $\left(\Psi_{s, s}\right)_{s \geq 0}$ is also predictable and $S$-integrable. Then we can write, as in (5.3),
\begin{equation*}
X_{t}=\int_{0}^{t} \Psi_{t, s} \mathrm{~d} S_{s}=\int_{0}^{t} \Psi_{s, s} \mathrm{~d} S_{s}+\int_{0}^{t}\left(\Psi_{t, s}-\Psi_{s, s}\right) \mathrm{d} S_{s}=: \int_{0}^{t} \Psi_{s, s} \mathrm{~d} S_{s}+Y_{t} \tag{5.6}
\end{equation*}
and the first $\mathrm{d} S$-integral on the right-hand side is well defined and a semimartingale. Note that if $S=M$ is a local martingale, that integral need not be a local martingale unless we know more about the diagonal $\left(\Psi_{s, s}\right)_{s \geq 0}$; see the classic example due to Emery [13]. We return to this point in Remark 5.2 below.

Now let us agree, as is natural in view of the problem, that $\Psi_{t, s}:=0$ for $s>t$, and also assume that for each $s \geq 0$, the process $\left(\Psi_{t, s}\right)_{t \geq 0}$ is right-continuous and of finite variation. This clearly generalises the assumption (5.2) or (5.4) in [28]. To apply our results from Section 4, we need to work with signed measures on a compact metric space. For each fixed $T \in(0, \infty)$, we therefore consider $\mathcal{K}_{T}:=[0, T]$ and associate to $\Psi$ for each $s \geq 0$ a "random signed $\mathbb{R}^{d}$-valued distribution function" $\bar{\varphi}_{s}$ on $[0, \infty)$ via
\begin{equation*}
t \mapsto \bar{\varphi}_{s}(t)(\omega):=I_{\{t \geq s\}}\left(\Psi_{t, s}(\omega)-\Psi_{s, s}(\omega)\right) \quad \text { for } s \geq 0 \text { and } t \geq 0 \text {. } \tag{5.7}
\end{equation*}
Restricted to $t \in \mathcal{K}_{T}=[0, T]$ (and hence also $s \leq T$ as $\left.s \leq t\right)$, the process $\bar{\varphi}^{(T)}=\left(\bar{\varphi}_{s}\right)_{0 \leq s \leq T}$ then induces an $\mathbb{M}^{d}\left(\mathcal{K}_{T}\right)$-valued process $\varphi^{(T)}=\left(\varphi_{s}^{(T)}\right)_{0 \leq s \leq T}$ which is weak* predictable because $s \mapsto \Psi_{t, s}$ is predictable. In fact,
\begin{equation*}
\varphi_{s}^{(T)}([0, t]):=\bar{\varphi}_{s}^{(T)}(t)=I_{\{t \geq s\}}\left(\Psi_{t, s}-\Psi_{s, s}\right) \quad \text { for } 0 \leq t \leq T \text { and } 0 \leq s \leq t \tag{5.8}
\end{equation*}
so we first get $\varphi_{s}^{(T)}(\{0\})=I_{\{0 \geq s\}}\left(\Psi_{0, s}-\Psi_{s, s}\right)=I_{\{s=0\}}\left(\Psi_{0,0}-\Psi_{0,0}\right)=0$ as $s \geq 0$, and then $s \mapsto \varphi_{s}^{(T)}((0, t])=\varphi_{s}^{(T)}([0, t])=I_{\{t \geq s\}}\left(\Psi_{t, s}-\Psi_{s, s}\right)$ is predictable by our assumptions. By\\
measure-theoretic induction, $\varphi^{(T)}(g)$ is then predictable for $g$ bounded and measurable on $\mathcal{K}_{T}$ and in particular for $g=f \in C\left(\mathcal{K}_{T}\right)$. Note that the variation norm of $\varphi_{s}^{(T)}$ is given by

$$
\left\|\varphi_{s}^{(T), i}\right\|_{\mathrm{var}}=\operatorname{var}_{t}\left(\left.\Psi_{\cdot, s}^{i}\right|_{[0, T]}\right)=\int_{s}^{T}\left|\Psi_{\mathrm{d} t, s}^{i}\right| \quad \text { for } i=1, \ldots, d
$$

With these preparations, we obtain the following result.

\begin{theorem} Suppose for each $T \in(0, \infty)$ and each $s \in[0, T]$ that the process $t \mapsto \Psi_{t, s}$ is of finite variation on $[0, T]$ and its $\left(\mathbb{R}^{d}\right.$-valued, componentwise) total variation process
\begin{equation*}
\left(\operatorname{var}_{t}\left(\Psi_{\cdot, s} \mid[0, T]\right)\right)_{0 \leq s \leq T}=\left(\int_{s}^{T}\left|\Psi_{\mathrm{d} t, s}^{i}\right|\right)_{0 \leq s \leq T}^{i=1, \ldots, d} \text { is } S \text {-integrable on }[0, T] \text {. } \tag{5.9}
\end{equation*}
Then the process $X$ from (5.1) can be written as
\begin{align*}
X_{t}=\int_{0}^{t} \Psi_{s, s} \mathrm{~d} S_{s}+Y_{t} & =\int_{0}^{t} \Psi_{s, s} \mathrm{~d} S_{s}+\int_{0}^{t}\left(\Psi_{t, s}-\Psi_{s, s}\right) \mathrm{d} S_{s}  \tag{5.10}\\
& =\int_{0}^{t} \Psi_{s, s} \mathrm{~d} S_{s}+\left(\varphi^{(T)} \bullet S_{T}\right)([0, t]), \quad 0 \leq t \leq T,
\end{align*}
for the $\mathbb{M}^{d}\left(\mathcal{K}_{T}\right)$-valued process $\varphi^{(T)}$ associated to $\Psi$. If each process $\varphi^{(T)} \bullet S$ is measurevalued (as opposed to only charge-valued), then $X$ is a semimartingale and $Y$ is predictable and of finite variation.\end{theorem}
\begin{proof}Proof. For each $T \in(0, \infty)$, we prove (5.10) by working with the compact set $\mathcal{K}_{T}:=[0, T]$. Results on $[0, \infty)$ are then obtained by pasting together.\\
1) Fix $T \in(0, \infty)$ and consider $\varphi^{(T)}$ as in (5.8). Using the definition of $Y$ in (5.6) and $\Psi_{t, s}=0$ for $s>t$ as well as (5.8), we can write, for $0 \leq t \leq T$,
\begin{align*}
Y_{t} & =\int_{0}^{t} I_{\{t \geq s\}}\left(\Psi_{t, s}-\Psi_{s, s}\right) \mathrm{d} S_{s}=\varphi^{(T)}\left(I_{[0, t]}\right) \cdot S_{t}  \tag{5.11}\\
& =\int_{0}^{T} \varphi_{s}^{(T)}([0, t]) \mathrm{d} S_{s}=\varphi^{(T)}\left(I_{[0, t]}\right) \cdot S_{T} . \tag{5.12}
\end{align*}
Thanks to the assumption (5.9), we can now invoke the weak* Fubini property (4.7) in Theorem 4.4 to obtain with the closed set $D:=[0, t] \subseteq \mathcal{K}_{T}$ that
\begin{equation*}
\left(\varphi^{(T)}\left(I_{[0, t]}\right)\right) \cdot S_{s}=\left(\varphi^{(T)} \bullet S_{s}\right)\left(I_{[0, t]}\right) \quad \text { for all } s \in[0, T] \tag{5.13}
\end{equation*}
giving for $s=T$ via (5.12) that
\begin{equation*}
Y_{t}=\left(\varphi^{(T)} \bullet S_{T}\right)\left(I_{[0, t]}\right) \quad \text { for } 0 \leq t \leq T \tag{5.14}
\end{equation*}
This proves (5.10). If $\varphi^{(T)} \bullet S$ has values in $\mathbb{M}\left(\mathcal{K}_{T}\right)$, then $Y=\left(Y_{t}\right)_{0 \leq t \leq T}$ is adapted by (5.11), and by (5.14) RCLL and of finite variation with respect to $t$, and hence in particular a semimartingale on $[0, T]$. So is then $X$ by (5.10) as $\left(\Psi_{s, s}\right)_{s \geq 0}$ is $S$-integrable.\\
2)  To show that $Y$ is even predictable, we first note that (5.8) shows that we have $\varphi_{s}^{(T)}\left(I_{[0, t]}\right)=\bar{\varphi}_{s}^{(T)}(t)=\varphi_{s}^{(T)}\left(I_{[0, t]}\right) I_{\{s<t\}}$ and therefore, by $(5.11)$,
\begin{equation*}
Y_{t}=\int_{0}^{t} \varphi_{s}^{(T)}\left(I_{[0, t]}\right) I_{\{s<t\}} \mathrm{d} S_{s}=\int_{0}^{t-} \varphi_{s}^{(T)}\left(I_{[0, t]}\right) \mathrm{d} S_{s}=\varphi^{(T)}\left(I_{[0, t]}\right) \cdot S_{t-\cdot} \tag{5.15}
\end{equation*}
On the other hand, the process $U(s, t):=\left(\varphi^{(T)} \bullet S_{t-}\right)\left(I_{[0, s]}\right)$ on $\mathcal{K}_{T} \times \bar{\Omega}$ is well defined, because $t \mapsto\left(\varphi^{(T)} \bullet S_{t}\right)\left(I_{[0, s]}\right)$ is by Theorem 4.4 in $\mathcal{R}^{0}$ and hence RCLL (with respect to $t$ ), and $U$ is $\mathcal{B}\left(\mathcal{K}_{T}\right) \otimes \mathcal{P}$-measurable as it is right-continuous in $s$ and left-continuous in $t$ and adapted to $\mathbb{F}$. Note that right-continuity in $s$ uses that $\varphi^{(T)} \bullet S$ is measure-valued. This implies by a monotone class argument that the diagonal of $U$ is predictable. Finally, (5.15) and (5.13) also yield that
$$
Y_{t}=\varphi^{(T)}\left(I_{[0, t]}\right) \cdot S_{t-}=\left(\varphi^{(T)} \bullet S_{t-}\right)\left(I_{[0, t]}\right)=U(t, t)
$$
so $Y$ is the diagonal of $U$ and therefore predictable. This completes the proof.\end{proof}
\begin{remark} The decomposition $X_{t}=\int_{0}^{t} \Psi_{s, s} \mathrm{~d} S_{s}+Y_{t}$ in (5.6) also allows us to say more about the structure and decomposition of the semimartingale $X$, again under all the assumptions of Theorem 5.1. We have just seen that $Y$ is predictable and of finite variation. The first summand in (5.6) is a stochastic integral, with respect to $S$, of the predictable $S$-integrable process $\left(\Psi_{s, s}\right)_{s \geq 0}$. Thus it is always a semimartingale, and we can work out its characteristics from the behaviour of $\Psi$ on the diagonal $\left(\Psi_{s, s}\right)_{s \geq 0}$ and from the characteristics of $S$ itself. If $S=M$ is a local martingale, we can also ask whether the integral is even a local martingale. (In that case, $X_{t}=\int_{0}^{t} \Psi_{t, s} \mathrm{~d} M_{s}, t \geq 0$, is a special semimartingale.) This depends on the interaction between the diagonal of $\Psi$ and the jumps of $M$; we need to check whether the increasing process $\left(\sum_{0 \leq s \leq}\left(\Psi_{s, s}^{\top} \Delta M_{s}\right)^{2}\right)^{1 / 2}$ is locally integrable. (Of course, this is trivially satisfied if $M$ is continuous.) For more details, see Jacod/Shiryaev [16, Remark III.6.28].\end{remark}
\section{Connections to the classic case}
In this section, we relate our results to some of the classic versions of stochastic Fubini theorems available in the literature. We do not aim for maximal generality, but rather explain the main conceptual differences.\\
We start again with an $\mathbb{R}^{d}$-valued semimartingale $S=\left(S_{t}\right)_{t \geq 0}$ on $(\Omega, \mathcal{F}, \mathbb{F}, P)$ and a compact metric space $\mathcal{K}$. We consider a fixed weak* predictable process $\varphi=\left(\varphi_{t}\right)_{t \geq 0}$ taking values in $\mathbb{M}^{d}(\mathcal{K})$ and write it in the form
\begin{equation*}
\varphi_{t}(\mathrm{~d} z)(\omega)=\psi_{t}(z)(\omega) \rho(\omega, t, \mathrm{~d} z) \tag{6.1}
\end{equation*}
for an $\mathbb{R}^{d}$-valued process $\psi$ on $\Omega \times(0, \infty) \times \mathcal{K}$ which is product-measurable and such that $\psi(z)$ is for each $z \in \mathcal{K}$ an $\mathbb{R}^{d}$-valued predictable process on $\bar{\Omega}$. In (6.1), $\rho(\omega, t, \mathrm{~d} z)$ is a transition kernel from $(\bar{\Omega}, \mathcal{P})$ to $(\mathcal{K}, \mathcal{B}(\mathcal{K}))$, i.e., $\rho(\cdot, \cdot, D)$ is a predictable process for each $D \in \mathcal{B}(\mathcal{K})$ and $\rho(\omega, t, \cdot)$ is a finite (nonnegative) measure on $\mathcal{B}(\mathcal{K})$ for each $(\omega, t) \in \bar{\Omega}$. This is not a restriction in generality for $\varphi$ as we can always take $\rho:=\sum_{i=1}^{d}\left|\varphi^{i}\right|$ and $\psi:=\frac{d \varphi}{d \rho}$. From (6.1) and using Cauchy-Schwarz, we obtain that $\|\varphi\|_{\text {var }}$ has the coordinates
\begin{equation*}
\left\|\varphi_{t}^{i}\right\|_{\mathrm{var}}=\int_{\mathcal{K}}\left|\psi_{t}^{i}(z)\right| \rho_{t}(\mathrm{~d} z) \leq\left(\rho_{t}(\mathcal{K}) \int_{\mathcal{K}}\left|\psi_{t}^{i}(z)\right|^{2} \rho_{t}(\mathrm{~d} z)\right)^{1 / 2} . \tag{6.2}
\end{equation*}
So $\|\varphi\|_{\text {var }}$ is $S$-integrable if (and only if) for some control process $V$ for $S$,
\begin{equation*}
\text { the process } \int \sum_{i=1}^{d}\left(\int_{\mathcal{K}}\left|\psi_{r}^{i}(z)\right| \rho_{r}(\mathrm{~d} z)\right)^{2} \mathrm{~d} V_{r} \text { is finite-valued, } \tag{6.3}
\end{equation*}
which is implied by the stronger condition that
\begin{equation*}
\text { the process } \int \rho_{r}(\mathcal{K})\left(\int_{\mathcal{K}}\left|\psi_{r}(z)\right|^{2} \rho_{r}(\mathrm{~d} z)\right) \mathrm{d} V_{r} \text { is finite-valued. } \tag{6.4}
\end{equation*}
A very special case of (6.1) arises if we assume that $\rho(\omega, t, \mathrm{~d} z) \equiv \eta(\mathrm{d} z)$ for some finite (nonnegative) measure $\eta$ on $\mathcal{B}(\mathcal{K})$. This means that the $(\omega, t)$-indexed family of measures $\varphi_{t}(\mathrm{~d} z)(\omega)$ is dominated by one fixed measure and we have
\begin{equation*}
\varphi_{t}(\mathrm{~d} z)(\omega)=\psi_{t}(z)(\omega) \eta(\mathrm{d} z) \tag{6.5}
\end{equation*}
where $\psi$ is a $z$-indexed family of predictable processes on $\bar{\Omega}$. Apart from slightly different conditions on $\mathcal{K}$ and $\eta$, this is the framework in which classic stochastic Fubini theorems are cast. More precisely, Protter [29, Section IV.6] and Veraar [31] both assume that ( $\mathcal{K}, \mathbb{K}$ ) is a measurable space and $\eta$ is a measure on $\mathbb{K}$, finite in [29] and $\sigma$-finite in [31]. Next, $S$ is in both $[29,31]$ a real-valued semimartingale, assumed continuous in [31], and $\psi$ is $\mathcal{P} \otimes \mathbb{K}$-measurable in [29] and $\rm{Prog}\otimes \mathbb{K}$-measurable in [31], where Prog is the progressive $\sigma$-field on $\bar{\Omega}$. The integrability condition on $\psi$ in [29] is that
\begin{equation*}
\text { the process }\left(\int_{\mathcal{K}}\left|\psi_{t}(z)\right|^{2} \eta(\mathrm{d} z)\right)_{t \geq 0}^{1 / 2} \text { is } S \text {-integrable, } \tag{6.6}
\end{equation*}
which by finiteness of $\eta$ is equivalent to the sufficient condition (6.4) above. If $S$ is continuous with canonical decomposition $S=S_{0}+M+A$, a sufficient condition for our weaker assumption (6.3) is by Shiryaev/Cherny [30, Definition 3.9 and part (ii) of the subsequent remark] that
\begin{equation*}
\text { the processes } \int\left(\int_{\mathcal{K}}\left|\psi_{r}(z)\right| \eta(\mathrm{d} z)\right) \mathrm{d} A_{r} \text { and } \int\left(\int_{\mathcal{K}}\left|\psi_{r}(z)\right| \eta(\mathrm{d} z)\right)^{2} \mathrm{~d}[M]_{r} \tag{6.7}
\end{equation*}
are both finite-valued.\\
The condition (6.7) is slightly different from the assumption in [31, Theorem 2.2, conditions (2.1) and (2.2)] that
$$
\text { the processes } \int_{\mathcal{K}}\left(\int\left|\psi_{r}(z)\right| \mathrm{d} A_{r}\right) \eta(\mathrm{d} z) \text { and } \int_{\mathcal{K}}\left(\int\left|\psi_{r}(z)\right|^{2} \mathrm{~d}[M]_{r}\right)^{1 / 2} \eta(\mathrm{d} z)
$$
are both finite-valued.\\
In the frameworks of either Protter [29, Section IV.6] or Veraar [31], the classic stochastic Fubini theorem says that in the dominated case (6.5), we have $P$-a.s. for all $t \geq 0$ that
\begin{equation*}
\int_{0}^{t}\left(\int_{\mathcal{K}} \psi_{r}(z) \eta(\mathrm{d} z)\right) \mathrm{d} S_{r}=\int_{\mathcal{K}}\left(\int_{0}^{t} \psi_{r}(z) \mathrm{d} S_{r}\right) \eta(\mathrm{d} z) \tag{6.8}
\end{equation*}
If we replace $\psi$ by $\psi I_{D}$ for a measurable subset $D$ of $\mathcal{K}$, this can be rewritten as
\begin{equation*}
\int_{0}^{t} \varphi_{r}(D) \mathrm{d} S_{r}=\int_{D}\left(\int_{0}^{t} \psi_{r}(z) \mathrm{d} S_{r}\right) \eta(\mathrm{d} z) \tag{6.9}
\end{equation*}
On the other hand, in our general setting and under (6.3), we obtain from Theorem 4.4 that for every closed $D \subseteq \mathcal{K}$, we get $P$-a.s. for all $t \geq 0$ that
\begin{equation*}
\int_{0}^{t} \varphi_{r}(D) \mathrm{d} S_{r}=\left(\varphi \bullet S_{t}\right)(D) \tag{6.10}
\end{equation*}
As the left-hand sides of (6.9) and (6.10) agree, so must the right-hand sides, whenever all the required assumptions are satisfied. Now by construction, the stochastic integral $\varphi \bullet S$ is in general only a charge-valued process; so the right-hand side of (6.10) need not be a measure in the argument $D$. On the other hand, the right-hand side of (6.9) can be expected to be a measure in the argument $D$, provided that the inner $(\mathrm{d} S$-)integral is an $\eta$-integrable function of $z$. So we can exploit the connection between (6.9) and (6.10) to get more information about $\varphi \bullet S$. The next result makes these ideas precise. Note that except for $\mathcal{K}$, its assumptions are precisely those of Protter [29, Theorem IV.65].
\begin{proposition} Suppose $S=\left(S_{t}\right)_{t \geq 0}$ is a real-valued semimartingale and $\mathcal{K}$ is a compact metric space with Borel $\sigma$-field $\mathcal{B}(\mathcal{K})$. Let $\varphi=\left(\varphi_{t}\right)_{t \geq 0}$ be a charge-valued process which satisfies the domination condition (6.5) so that $\varphi_{t}(\mathrm{~d} z)(\omega)=\psi_{t}(z)(\omega) \eta(\mathrm{d} z)$ for a finite (nonnegative) measure $\eta$ on $\mathcal{B}(\mathcal{K})$. Suppose that $\psi$ is $\mathcal{P} \otimes \mathcal{B}(\mathcal{K})$-measurable and satisfies (6.6). Then the stochastic integral process $\varphi \bullet S$ is well defined and takes values in $\mathbb{M}$. In other words, our stochastic integral is then not only charge-valued, but measure-valued.\end{proposition}
\begin{proof}1) We first argue that $\varphi$ is $\mathbb{M}$-valued, weak${ }^{*}$ predictable and in $\Phi$. As in (6.2),
\begin{equation*}
\left\|\varphi_{t}\right\|_{\text {var }}=\left|\varphi_{t}\right|(\mathcal{K})=\int_{\mathcal{K}}\left|\psi_{t}(z)\right| \eta(\mathrm{d} z) \leq\left(\eta(\mathcal{K}) \int_{\mathcal{K}}\left|\psi_{t}(z)\right|^{2} \eta(\mathrm{d} z)\right)^{1 / 2} . \tag{6.11}
\end{equation*}
Thanks to (6.6), the right-hand side has a finite integral with respect to some control process $V$ for $S$, and if we choose (as we can) $V$ strictly increasing, the left-hand side of (6.11) must be finite $P$-a.s. for all $t \geq 0$ like the right-hand side. So $\varphi$ is $\mathbb{M}$-valued, and $\|\varphi\|_{\text {var }}$ is $S$-integrable again due to (6.6). For each $z$, the $z$-section $\psi(z)$ is predictable because $\psi$ is $\mathcal{P} \otimes \mathcal{B}(\mathcal{K})$-measurable, and for any $f \in C(\mathcal{K})$,
$$
\int\left|f(z) \psi_{t}(z)\right| \eta(\mathrm{d} z) \leq\|f\|_{\infty} \int_{\mathcal{K}}\left|\psi_{t}(z)\right| \eta(\mathrm{d} z)<\infty \quad P \text {-a.s. for all } t \geq 0
$$
shows that $\varphi(f)$ is well defined and hence a predictable process. So $\varphi$ is also weak* predictable and therefore in $\Phi$.\\
 2) Thanks to 1), Theorem 4.4 tells us that $\varphi \bullet S$ is well defined and that for every closed set $D \subseteq \mathcal{K}$,
\begin{equation*}
(\varphi \bullet S)(D)=\varphi(D) \cdot S=\int \varphi_{r}(D) \mathrm{d} S_{r}=\int\left(\int_{D} \psi_{r}(z) \eta(\mathrm{d} z)\right) \mathrm{d} S_{r} \tag{6.12}
\end{equation*}
On the other hand, the classic stochastic Fubini theorem in Protter [29, Theorem IV.65] (applied to $\psi I_{D}$ ) tells us that
\begin{equation*}
\int\left(\int_{D} \psi_{r}(z) \eta(\mathrm{d} z)\right) \mathrm{d} S_{r}=\int_{D}\left(\int \psi_{r}(z) \mathrm{d} S_{r}\right) \eta(\mathrm{d} z) \tag{6.13}
\end{equation*}
and that $\int \psi_{r}(z) \mathrm{d} S_{r}$ can be chosen product-measurable on $\bar{\Omega} \times \mathcal{K}$ so that $z \mapsto \int \psi_{r}(z) \mathrm{d} S_{r}$ is then $\mathcal{B}(\mathcal{K})$-measurable. So the right-hand side of (6.13), and hence by $(6.12)$ also $(\varphi \bullet S)(D)$, is a finite signed measure as a function of $D$ if we can show that
\begin{equation*}
\int_{\mathcal{K}}\left|\int_{0}^{t} \psi_{r}(z) \mathrm{d} S_{r}\right| \eta(\mathrm{d} z)<\infty \quad P \text {-a.s. for each } t \geq 0 \tag{6.14}
\end{equation*}
3)  By the assumption (6.6), the process $U:=\left(\int_{\mathcal{K}}|\psi(z)|^{2} \eta(\mathrm{d} z)\right)^{1 / 2}$ is $S$-integrable. By Shiryaev/Cherny [30, Definition 3.9; see also Lemma 4.11], there exists a decomposition $S=S_{0}+M+A$ into a local martingale $M$ and an adapted RCLL process $A$ of finite variation $|A|$, both null at 0 , such that $U$ is in both $L_{\mathrm{var}}(A)$ and $L_{\mathrm{loc}}^{1}(M)$. By the definition of $L_{\mathrm{var}}(A)$ in [30], this implies via Fubini-Tonelli and Cauchy-Schwarz that
\begin{align*}
\int_{\mathcal{K}}\left(\left|\int_{0}^{t} \psi_{r}(z) \mathrm{d} A_{r}\right|\right) \eta(\mathrm{d} z) & \leq \int_{\mathcal{K}}\left(\int_{0}^{t}\left|\psi_{r}(z)\right| \mathrm{d}|A|_{r}\right) \eta(\mathrm{d} z)  \tag{6.15}\\
& =\int_{0}^{t}\left(\int_{\mathcal{K}}\left|\psi_{r}(z)\right| \eta(\mathrm{d} z)\right) \mathrm{d}|A|_{r} \\
& \leq \int_{0}^{t}(\eta(\mathcal{K}))^{1 / 2} U_{r} \mathrm{~d}|A|_{r}<\infty \quad P \text {-a.s. }
\end{align*}
Next, the property $U \in L_{\text {loc }}^{1}(M)$ means that $\left(\int U^{2} \mathrm{~d}[M]\right)^{1 / 2}$ is locally integrable so that\\
along a localizing sequence $\left(\tau_{n}\right)_{n \in \mathbb{N}}$ of stopping times, we have by Fubini-Tonelli that
$$
\begin{aligned}
E\left[\left(\int_{\mathcal{K}}\left(\int_{0}^{\tau_{n}}\left|\psi_{r}(z)\right|^{2} \mathrm{~d}[M]_{r}\right) \eta(\mathrm{d} z)\right)^{1 / 2}\right] & =E\left[\left(\int_{0}^{\tau_{n}}\left(\int_{\mathcal{K}}\left|\psi_{r}(z)\right|^{2} \eta(\mathrm{d} z)\right) \mathrm{d}[M]_{r}\right)^{1 / 2}\right] \\
& =E\left[\left(\int_{0}^{\tau_{n}} U_{r}^{2} \mathrm{~d}[M]_{r}\right)^{1 / 2}\right]<\infty .
\end{aligned}
$$
By Cauchy-Schwarz,
$$
\int_{\mathcal{K}}\left(\int\left|\psi_{r}(z)\right|^{2} \mathrm{~d}[M]_{r}\right)^{1 / 2} \eta(\mathrm{d} z) \leq(\eta(\mathcal{K}))^{1 / 2}\left(\int_{\mathcal{K}}\left(\int\left|\psi_{r}(z)\right|^{2} \mathrm{~d}[M]_{r}\right) \eta(\mathrm{d} z)\right)^{1 / 2}
$$
and so using Fubini-Tonelli again yields that also
$$
\int_{\mathcal{K}} E\left[\left(\int_{0}^{\tau_{n}}\left|\psi_{r}(z)\right|^{2} \mathrm{~d}[M]_{r}\right)^{1 / 2}\right] \eta(\mathrm{d} z)=E\left[\int_{\mathcal{K}}\left(\int_{0}^{\tau_{n}}\left|\psi_{r}(z)\right|^{2} \mathrm{~d}[M]_{r}\right)^{1 / 2} \eta(\mathrm{d} z)\right]<\infty
$$
for each $n$. By the Burkholder-Davis-Gundy inequality,
$$
E\left[\left(\int \psi(z) \mathrm{d} M\right)_{\tau_{n}}^{*}\right] \leq \text { const. } E\left[\left(\int_{0}^{\tau_{n}}\left|\psi_{r}(z)\right|^{2} \mathrm{~d}[M]_{r}\right)^{1 / 2}\right]
$$
for each $z$, where the constant does not depend on $z$, and so
$$
\begin{aligned}
E\left[\int_{\mathcal{K}}\left|\int_{0}^{\tau_{n}} \psi_{r}(z) \mathrm{d} M_{r}\right| \eta(\mathrm{d} z)\right] & \leq E\left[\int_{\mathcal{K}}\left(\int \psi(z) \mathrm{d} M\right)_{\tau_{n}}^{*} \eta(\mathrm{d} z)\right] \\
& =\int_{\mathcal{K}} E\left[\left(\int \psi(z) \mathrm{d} M\right)_{\tau_{n}}^{*}\right] \eta(\mathrm{d} z) \\
& \leq \text { const. } \int_{\mathcal{K}} E\left[\left(\int_{0}^{\tau_{n}}\left|\psi_{r}(z)\right|^{2} \mathrm{~d}[M]_{r}\right)^{1 / 2}\right] \eta(\mathrm{d} z)<\infty
\end{aligned}
$$
for each $n$. This implies that
$$
\int_{\mathcal{K}}\left|\int_{0}^{t} \psi_{r}(z) \mathrm{d} M_{r}\right| \eta(\mathrm{d} z)<\infty \quad P \text {-a.s. on each } \llbracket 0, \tau_{n} \rrbracket \text {, }
$$
and so (6.14) follows by combining this with (6.15). This completes the proof.\end{proof}
If we compare the classic stochastic Fubini theorem to our new result in Theorem 4.4, we can see several differences. First, the former needs stronger assumptions on the integrand $\varphi($ or $\psi)$ than the latter; this can also be seen in the proof of Proposition 6.1. Second, if we represent $\varphi$ as in (6.1) as $\varphi_{t}(\mathrm{~d} z)=\psi_{t}(z) \rho(\omega, t, \mathrm{~d} z)$, the general weak* Fubini property (4.7) in Theorem 4.4 for $D=\mathcal{K}$ (with the two sides interchanged) takes the form
\begin{equation*}
\int_{0}^{t}\left(\int_{\mathcal{K}} \psi_{r}(z) \rho_{r}(\mathrm{~d} z)\right) \mathrm{d} S_{r}=\int_{\mathcal{K}}\left(\int_{0}^{t} \psi_{r}(z) \rho_{r}(\cdot) \mathrm{d} S_{r}\right)(\mathrm{d} z) \tag{6.16}
\end{equation*}
In comparison to the classic result (6.8), the left-hand side above is the obvious generalisation where we now mix the integrands $\psi_{r}(z)$ with $\rho_{r}(\mathrm{~d} z)$ instead of only $\eta(\mathrm{d} z)$. However, the right-hand side of (6.16) is more complicated than in (6.8); we really need the stochastic integral with respect to $S$ of the measure-valued process $\psi \rho$, and it is no longer directly possible to rewrite or construct this explicitly as a mixture (over $z$ ) of stochastic integrals $\int \psi_{r}(z) \mathrm{d} S_{r}$. This is because in contrast to the fixed measure $\eta$, the kernel $\rho$ depends on both $\omega$ and $t$ in general.\\
In comparison to the existing literature on stochastic Fubini theorems, our main contribution is twofold. At the technical level, we deal with a parametric family of integrands that is not dominated, with respect to its parameter, by a single and nonrandom measure. At the conceptual level, we achieve our results via the idea of viewing integrands as measure-valued processes and developing and using a corresponding stochastic integration theory. The latter viewpoint is similar in spirit to the approach of Bichteler and co-authors and has also been used in van Neerven/Veraar [23]; but its implementation must be done here more generally because we do not have a dominating measure $\eta$, and we also need to develop a concept of stochastic integration of measure-valued integrands with respect to general semimartingales.
\begin{remark} It is an interesting question whether our approach allows to obtain the classic stochastic Fubini theorem as a special case, and maybe even under weaker assumptions. Purely formally, one could start from (6.16), plug in $\eta$ for $\rho_{r}$ and then hope that one could move the deterministic $\eta$ outside to $\mathrm{d} z$ on the right-hand side of (6.16). This would also imply that in the dominated setting (6.5), our stochastic integral process $\varphi \bullet S$ is dominated\\
by the measure $\eta$. If $\varphi \in \mathcal{E}$ is an elementary integrand, the above results can be proved easily because all integrals then reduce to finite sums. But in general, one would probably need a similar approximation as in the key Lemma 3.6, now at the level of densities $\psi$, and it seems very likely that this would require extra conditions on $\psi$, maybe similar to the condition (6.6) in Protter [29, Theorem IV.65]. We leave this issue as an open problem and refer to Section 7 for a caveat against too much optimism.\end{remark}
\section{An illustrative example}
On the filtered probability space $(\Omega, \mathcal{F}, \mathbb{F}, P)$, let $S=W=\left(W_{t}\right)_{t \geq 0}$ be a Brownian motion. Fix $T \in(0, \infty)$ and take the compact set $\mathcal{K}=\mathcal{K}_{T}=[0, T]$. For $\alpha>0$, define
\begin{equation*}
\varphi_{t}(\mathrm{~d} z):=\alpha(z-t)^{\alpha-1} I_{\{z>t\}} \mathrm{d} z=: \psi_{t}(z) \mathrm{d} z \quad \text { for } t \geq 0 \text { and } z \in \mathcal{K} \tag{7.1}
\end{equation*}
Then each $\varphi_{t}$ is a (nonnegative) measure on $\mathcal{B}(\mathcal{K})$ with
\begin{equation*}
\left\|\varphi_{t}\right\|_{\mathrm{var}}=\varphi_{t}(\mathcal{K})=\int_{t}^{T} \alpha(z-t)^{\alpha-1} \mathrm{~d} z=(T-t)^{\alpha}<\infty \tag{7.2}
\end{equation*}
and so $\varphi=\left(\varphi_{t}\right)_{t \geq 0}$ is an $\mathbb{M}$-valued process. Moreover, as it is nonrandom, the process $\varphi$ is clearly weak* predictable. Note that $\psi_{t}(z)=0$ for $z \in \mathcal{K}=[0, T]$ and $t>T$; so it is enough to look only at $\left(\varphi_{t}\right)_{0 \leq t \leq T}$ and $\left(W_{t}\right)_{0 \leq t \leq T}$.
If we take for $W$ the control process $V_{t}=t$, then (7.2) gives
$$
\mathcal{D}_{t}\left(\|\varphi\|_{\mathrm{var}} ; V\right)=\int_{0}^{t}\left\|\varphi_{r}\right\|_{\mathrm{var}}^{2} \mathrm{~d} V_{r}=\int_{0}^{t}(T-r)^{2 \alpha} \mathrm{d} r=\frac{1}{2 \alpha+1}\left(T^{2 \alpha+1}-(T-t)^{2 \alpha+1}\right)
$$
This is finite for any $t \in[0, T]$ and $\alpha>0$ so that $\varphi$ is in $\Phi$. Thus $\varphi \bullet W$ is well defined (on $[0, T]$ ) and we can use the general weak* Fubini property (4.6) from Theorem 4.4. In particular, taking $D=[0, u]$ and computing $\varphi_{r}([0, u])=(u-r)^{\alpha} I_{\{u>r\}}$ as in (7.2), we get
$$
\left(\varphi \bullet S_{T}\right)([0, u])=\varphi([0, u]) \cdot S_{T}=\int_{0}^{T} I_{\{u>r\}}(u-r)^{\alpha} \mathrm{d} S_{r}=\int_{0}^{u}(u-r)^{\alpha} \mathrm{d} W_{r}
$$
To check whether $\varphi \bullet S$ is a measure-valued process, we want to use Proposition 6.1. From (7.1), it is clear that we are in the dominated case (6.5), with $\eta$ being Lebesgue\\
measure on $\mathcal{K}=[0, T]$, and $\varphi$ is clearly $\mathcal{P} \otimes \mathcal{B}(\mathcal{K})$-measurable. For the integrability condition (6.6), we need to look at
\begin{align*}
\int_{0}^{t}\left(\int_{\mathcal{K}}\left|\psi_{r}(z)\right| \eta(\mathrm{d} z)\right) \mathrm{d} V_{r} & =\int_{0}^{t} \int_{r}^{T} \alpha^{2}(z-r)^{2 \alpha-2} \mathrm{~d} z \mathrm{~d} r  \tag{7.3}\\
& =\frac{\alpha^{2}}{2 \alpha-1} \int_{0}^{t}(T-r)^{2 \alpha-1} \mathrm{~d} r \\
& =\frac{\alpha^{2}}{2 \alpha(2 \alpha-1)}\left(T^{2 \alpha}-(T-t)^{2 \alpha}\right)
\end{align*}
where the second equality requires $2 \alpha-2>-1$ or $\alpha>\frac{1}{2}$ to ensure that the inner integral does not diverge at $z=r$. For $\alpha>\frac{1}{2}$, the quantity in (7.3) is always finite; so (6.6) is then satisfied and Proposition 6.1 tells us that $\varphi \bullet S$ is for $\alpha>\frac{1}{2}$ a measure-valued process.\\
To get more information about $\varphi \bullet S$, we now want to use the results on Volterra-type semimartingales from Section 5. As in (5.7) and (5.4), we look at $\Psi$ given by $\Psi_{s, s}=0$ and $\Psi_{t, s} I_{\{t \geq s\}}=I_{\{t \geq s\}} \varphi_{s}((0, t])=I_{\{t \geq s\}} \int_{0}^{t} \psi_{s}(z) \mathrm{d} z=I_{\{t \geq s\}} \int_{s}^{t} \alpha(z-s)^{\alpha-1} \mathrm{~d} z=I_{\{t \geq s\}}(t-s)^{\alpha}$. Clearly $t \mapsto \Psi_{t, s}$ is of finite variation, and it also satisfies (5.9) due to (7.2) and because $\varphi \in \Phi$. We then obtain as in (5.6) that for any $\alpha>0$,
\begin{equation*}
Y_{u}=\int_{0}^{u}\left(\Psi_{u, s}-\Psi_{s, s}\right) \mathrm{d} S_{s}=\int_{0}^{u}(u-s)^{\alpha} \mathrm{d} W_{s} \tag{7.4}
\end{equation*}
and Theorem 5.1 tells us that $Y=\left(Y_{u}\right)_{0 \leq u \leq T}$ is a semimartingale if $\varphi \bullet S$ is measurevalued. But the process $Y$ in (7.4) has been studied (even for $W$ replaced by a Lévy process) in detail in Basse/Pedersen [3], and [3, Theorem 3.1], with $\sigma=1$, shows that $Y$ is a semimartingale (in the filtration of $W$ ) if and only if the function $x \mapsto f(x)=x^{\alpha}$ has $f^{\prime}$ locally square-integrable on $[0, \infty)$, which holds as in (7.3) if and only if $\alpha>\frac{1}{2}$.
\begin{remark} In this very specific example, we see that $\varphi \bullet S$ is a measure-valued process if and only if a Volterra-type process $\left(\int_{0}^{t} \Psi_{t, s} \mathrm{~d} S_{s}\right)_{t \geq 0}$ with $\Psi$ naturally associated to $\varphi$ is a semimartingale. It is very tempting to conjecture that such a result could hold more generally, maybe even beyond the dominated setting (6.5). But even formulating this precisely, let alone proving it, is beyond the scope of the present paper.\end{remark}
\noindent{\bf Acknowledgement:} Martin Schweizer gratefully acknowledges financial support by the Swiss Finance Institute (SFI). Tahir Choulli gratefully acknowledges financial support by NSERC (the Natural Sciences and Engineering Research Council of Canada, Grant G121210818). Moreover, financial support by the Forschungsinstitut f\"ur Mathematik (FIM) at ETH Z\^urich is gratefully acknowledged. It is a pleasure to thank Thomas Hille for a number of useful technical remarks and discussions on measurability, Alex Schied for some very useful comments, and Mark Veraar for promptly and helpfully answering questions on the related literature.


\end{document}